\documentclass{amsart}
\usepackage{graphicx} 
\usepackage{amsthm}
\usepackage{amsmath}
\usepackage{amssymb}
\usepackage{hyperref}
\usepackage{mathtools}
\usepackage{setspace}
\usepackage{booktabs}
\usepackage{float}
\theoremstyle{plain}
\newtheorem{theorem}{Theorem}[section]
\theoremstyle{definition}
\newtheorem{Def}[theorem]{Definition}
\theoremstyle{plain}
\newtheorem{prop}[theorem]{Proposition}
\theoremstyle{plain}
\newtheorem{corollary}[theorem]{Corollary}
\theoremstyle{plain}
\newtheorem{lemma}[theorem]{Lemma}
\theoremstyle{remark}
\newtheorem{remark}[theorem]{Remark}
\theoremstyle{remark}
\newtheorem{example}[theorem]{Example}
\theoremstyle{plain}

\theoremstyle{definition}
\newtheorem{question}[theorem]{Question}
\theoremstyle{plain}

\usepackage[margin=1.2in]{geometry}
\usepackage{tikz-cd}
\usepackage{orcidlink}
\DeclareMathOperator{\Gal}{\mathrm{Gal}}
\DeclareMathOperator{\Res}{\mathrm{Res}}
\DeclareMathOperator{\Cov}{\mathrm{Cov}}
\DeclareMathOperator{\loc}{\mathrm{loc}}
\DeclareMathOperator{\lk}{\mathrm{lk}}
\DeclareMathOperator{\Spec}{\mathrm{Spec}}
\DeclareMathOperator{\im}{\mathrm{im}}
\DeclareMathOperator{\Frob}{\mathrm{Frob}}
\DeclareMathOperator{\coker}{\mathrm{coker}}
\DeclareMathOperator{\Hom}{\mathrm{Hom}}
\DeclareMathOperator{\et}{\acute et}
\newcounter{main}

\theoremstyle{plain}
\newtheorem{mainthm}[main]{\textbf{Theorem}}
\theoremstyle{plain}

\newtheorem*{theorem*}{Theorem}
\def\doubleprod#1#2{\ooalign{$#1\prod$\cr$#1\coprod$\cr}}
\DeclareMathOperator*{\Rprod}{\mathpalette\doubleprod\relax}

\usepackage[pagewise, mathlines]{lineno} 
\usepackage{longtable,caption}
\numberwithin{equation}{section}

\subjclass[2020]{57M12, 11R32, 11S25, 11S20, 11R42, 20E18, 18E50}
\keywords{Knot, 3-manifold, Chebotarev law, Galois cohomology, Neukirch--Uchida theorem, anabelian geometry, profinite rigidity, arithmetic topology}

\title{Rigidity and Galois theory of 3-manifolds branched over Chebotarev links}
\author{Nadav Gropper}
\address{Department of Mathematics, Seoul National University, Seoul, South Korea}
\email{Nadav.Gropper@gmail.com \orcidlink{0009000913557848}}

 \author{Jun Ueki}
\address{Department of Mathematics, Ochanomizu University; 2-1-1 Otsuka, Bunkyo-ku, 112-8610, Tokyo, Japan} 
\email{uekijun46@gmail.com \orcidlink{0000000337991074}}

\author{Yi Wang}
\address{Department of Mathematics, University of Illinois Urbana-Champaign, Urbana, USA}
\email{yiwang20@illinois.edu \orcidlink{0009000597676684}}

\date{\today}

\begin{document}

\begin{abstract}
The classical Neukirch--Uchida theorem states that the absolute Galois group determines a number field up to isomorphism. We prove an analogue of this theorem for 3-manifolds in the framework of arithmetic topology. We study infinite links $\mathcal{L}$ in 3-manifolds that satisfy a Chebotarev density property, and hence behave like the set of primes. We define the absolute Galois group of a 3-manifold relative to such a link as the inverse limit of profinite completions of finite sublink complements. Our main result shows that two covers of $S^3$ branched over finite sublinks of $\mathcal{L}$ are isomorphic if and only if their absolute Galois groups relative to $\mathcal{L}$ are isomorphic via a characteristic-preserving isomorphism. The proof translates the key ideas from the number-theoretic argument into 3-manifolds, developing Hilbert ramification theory for infinite covers and local-global principles. Doing so also provides a systematic justification for viewing Chebotarev links as topological analogues of the set of prime numbers. In addition, we propose further conditions on Chebotarev links which would further strengthen their connection to the prime numbers. 
\end{abstract} 

\maketitle

\tableofcontents
\newpage
\section{Introduction}

The main goal of this paper is to import tools from anabelian geometry to the study of 3-manifolds in the spirit of arithmetic topology. In particular, we develop  a rigidity result for branched covers of $S^3$ with respect to profinite groups, which is a topological analogue of the \emph{Neukirch--Uchida theorem}.  

\subsection{Rigidity of number fields and 3-manifolds}

In number theory, the Neukirch--Uchida theorem is a foundational result in anabelian geometry; it is a rigidity result for number fields with respect to their \'etale fundamental groups. Let $\overline{F}$ denote an algebraic closure of a number field $F$. 
Recall that the \emph{absolute Galois group} $\Gal(\overline{F}/F)$ is a profinite group. Given two number fields $F_1, F_2$ and an isomorphism $\alpha: \overline{F_1} \to \overline{F_2}$ such that $\alpha(F_1) = F_2$, there is an induced isomorphism $\alpha^\ast:\Gal(\overline{F_1}/F_1) \to \Gal(\overline{F_2}/F_2)$ defined by
\begin{equation}\label{eq:induced}
    \alpha^*(g)(x) = \alpha(g(\alpha^{-1}(x))) \ \ \ \ \ g \in \Gal(\overline{F_1}/F_1), x \in \overline{F_2}.
\end{equation}
i.e. $\alpha^* = \alpha \circ g \circ \alpha^{-1}$. The Neukirch--Uchida theorem asserts the converse:

\begin{theorem*}[The Neukirch--Uchida theorem \cite{neukirchcrelle, neukirchinvenciones, uchida}]\label{thm:NU}
Let $F_1, F_2$ be number fields, and let $\sigma: \Gal(\overline{F_1}/F_1) \to \Gal(\overline{F_2}/F_2)$ be an isomorphism of profinite groups. Then there exists a unique isomorphism $\alpha: \overline{F_1} \to \overline{F_2}$ for which $\alpha(F_1)=F_2$ which induces $\sigma$, i.e., $\sigma = \alpha^*$. 
\end{theorem*}

Arithmetic topology, which originated from observations of Mazur \cite{mazur2} and was later expanded upon in \cite{mazur1973notes}, pursues analogies between the behavior of knots in 3-manifolds and primes in number fields. For a comprehensive overview on arithmetic topology, see Morishita \cite{morishita2012knots}. Roughly, the classical arithmetic topology dictionary states:
\[
\begin{array}{ccc}
\text{\ensuremath{\Spec(\mathbb{Z})\cup\{\infty\}}} & \leftrightsquigarrow & \text{3-sphere \ensuremath{S^{3}}}\\ \text{ Number ring $\Spec(\mathcal{O}_K)\cup\{\text{infinite places}\}$} & \leftrightsquigarrow & \text{ closed 3-manifold}\\
\text{Set of primes \ensuremath{S=\{p_1,\dots,p_n\}}} & \leftrightsquigarrow & \text{Link \ensuremath{L=\cup K_i\subset S^{3}}}\\
\text{\ensuremath{\mathbb{Q}_{p}}} & \leftrightsquigarrow & \text{Boundary of a tubular neighbourhood of \ensuremath{K}}\\
\text{\ensuremath{\mathcal{O}_{K,S}}} & \leftrightsquigarrow & \text{complement of link $M\setminus L$.}
\end{array}
\]
A further analogy relevant to this paper is the analogue of Hilbert ramification theory presented in Chapter 5 of Morishita \cite{morishita2012knots} for $S^3$ and developed further by Ueki \cite{uekibranched} for general 3-manifolds. 

\medskip

Under the above analogies, the Neukirch--Uchida theorem is said to be an analogue of the following:

\begin{theorem*}[Mostow's rigidity theorem \cite{mostow}]
Let $M, N$ be two finite-volume hyperbolic 3-manifolds such that $\pi_1(M) \cong \pi_1(N)$. Then $M$ and $N$ are isometric.
\end{theorem*}

There are two big differences between the Mostow rigidity theorem and the Neukirch--Uchida theorem:
\begin{itemize}
    \item For a finite set of primes $S$, let $\mathbb{Z}_S = \mathbb{Z}[1/p\mid p\in S]$. Then $\Spec(\mathbb{Z}_S)$ is analogous to a hyperbolic 3-manifold with finitely many torus boundaries. To get a number field from its ring of integers, one needs to invert the set of all primes $S$ (i.e., an infinite set); for instance, $\mathbb{Q}$ can be viewed as $\mathbb{Z}_S$ where $S$ is the set of all prime integers. As opposed to the Neukirch--Uchida theorem, which deals with inverting an infinite set, the Mostow rigidity theorem deals with hyperbolic manifolds with finitely many boundary components (cusps).
    \item As opposed to the discrete fundamental groups in Mostow rigidity, Galois groups are always profinite groups, so rigidity results for number rings and fields with respect to $\pi_1^{\et}(\Spec(\mathbb{Z}_S))$ and $\Gal(\overline{\mathbb{Q}}/\mathbb{Q})$ are analogous rigidity results of 3-manifolds with respect to a profinite group.
\end{itemize}

Our main theorem addresses the two points above, giving an analogue of the Neukirch--Uchida theorem for 3-manifolds which involves countably infinite boundary components and profinite groups.

\subsection{Main results}

A major question of arithmetic topology is determining which sets of knots are analogous to the set of primes in number fields, especially when dealing with infinite countable sets of knots. Mazur \cite{mazur} proposed a property for links with countably infinite components, which is an analogue of the Chebotarev density theorem. An example of links with such a property was given by McMullen \cite{mcmullen2013knots}. Ueki \cite{ueki2021chebotarev} later defined \emph{stably Chebotarev} links \footnote{see Remark \ref{rmk:stably}}, which are Chebotarev links for which any preimage of the link in a branched cover satisfies the Chebotarev property; examples of stably Chebotarev links can be found in \cite{mcmullen2013knots} and \cite{ueki2021chebotarev}. 

\medskip

Throughout, we fix a stably Chebotarev link $\mathcal{L}$ in $S^3$, together with a fixed ordering of $\mathcal{L}$.

\medskip

We consider the category of all finite branched covers with base points of $S^3$ branched along a finite sublink of $\mathcal{L}$, which we denote by ${\rm Cov}(S^3,\mathcal{L})$ (see Section \ref{sec:background} for more detailed information). This forms a Galois category, and we define the absolute Galois group to be the fundamental group of this category. This group is pro-represented by a universal pro-cover, which we denote by $\overline{S^3_\mathcal{L}}\to S^3$. For any finite branched cover $h:M\rightarrow S^3$ in $\Cov(S^3, \mathcal{L})$, we take a link in $M$, denoted by $\mathcal{L}_M:=h^{-1}(\mathcal{L})$ with an induced ordering (see Section \ref{sec:chebo}).

\medskip

We have the following equivalent characterizations of the absolute Galois group (Proposition \ref{prop:Galoisgroups}):

\begin{Def}\label{def:galois}
The \emph{absolute Galois group} of $(M, \mathcal{L}_M)$ is defined as 
\begin{equation}
    \Gal(M, \mathcal{L}_M) = \varprojlim_{h \in \Cov(M, \mathcal{L}_M)}\Gal(h)= \varprojlim_{L\subset \mathcal{L}_M} \widehat{\pi}_1(M\setminus L)
\end{equation}
where $L$ runs through finite sublinks of $\mathcal{L}_M$.
\end{Def}

Thus, one way to think of the absolute Galois group of $M$ relative to $\mathcal{L}_M$ is the inverse limit of the profinite completions of fundamental groups, where the limit is taken by deleting increasingly more components of $\mathcal{L}_M$. Note that $\Gal(M, \mathcal L_M)$ is identified with an open subgroup of $\Gal(S^3, \mathcal L)$. Each component $K \subset \mathcal L_M$, together with a compatible system $\overline K$ of components above it in the universal pro-cover, has an \emph{absolute decomposition group} $D_{\overline K|K} \subset \Gal(M, \mathcal L_M)$. Theorem \ref{thm:profinitetorus} identifies $D_{\overline K|K}$ with the profinite completion of the peripheral subgroup $\pi_1(\partial V_K)$. Its meridional subgroup is the \emph{absolute inertia group}, and the longitude class plays the role of the Frobenius.

\begin{Def}\label{def:charpres}
Given two branched covers $h_1:M_1\rightarrow S^3$, $h_2:M_2\rightarrow S^3$ in $\Cov(S^3, \mathcal{L})$, we say that a bijection $\sigma_\mathcal{L}:\mathcal{L}_{M_1}\rightarrow\mathcal{L}_{M_2}$  \emph{preserves characteristic} if $h_2(\sigma_\mathcal{L}(K))=h_1(K)$ for all knots $K\in\mathcal{L}_{M_1}$. If a group isomorphism $\sigma: \Gal(M_1, \mathcal{L}_{M_1}) \to \Gal(M_2, \mathcal{L}_{M_2})$ induces a bijection $\sigma_\mathcal{L}: \mathcal{L}_{M_1} \to \mathcal{L}_{M_2}$ as in Theorem \ref{thm:main}(1), we say that $\sigma$ \emph{preserves characteristic} if $\sigma_\mathcal{L}$  preserves characteristic. 
\end{Def}

Let $U_i = \Gal(M_i, \mathcal L_{M_i})$. An isomorphism of branched covers $M_1 \to M_2$ over $S^3$ lifts to an element $\alpha \in \Gal(S^3, \mathcal L)$ such that $\alpha U_1\alpha^{-1} = U_2$. Conversely, such an $\alpha$ maps the subcover $M_1$ to $M_2$ and descends to an isomorphism of branched covers $\alpha|_{M_1}: M_1 \to M_2$ over $S^3$. The element $\alpha$ induces $\alpha^*: U_1 \to U_2$, $\alpha^*(g) = \alpha g \alpha^{-1}$, and the lifts of a given isomorphism $M_1 \to M_2$ differ by elements of $U_1$, so $\alpha^*$ is determined by $\alpha|_{M_1}$ up to precomposition with an inner automorphism of $U_1$. We say that an isomorphism $\sigma: U_1 \to U_2$ is \emph{induced by} an isomorphism of branched covers if $\sigma = \alpha^*$ (up to an inner automorphism) for some lift $\alpha$ of that isomorphism.

\begin{remark}\label{rem:inducedIsom}
From the definition of characteristic-preserving, one sees that any isomorphism of branched covers preserves characteristic, i.e., if $\alpha: M_1 \to M_2$ is an isomorphism of covers over $S^3$, then the induced isomorphism $\alpha_*: \Gal(M_1, \mathcal{L}_{M_1}) \to \Gal(M_2, \mathcal{L}_{M_2})$ preserves characteristic.
\end{remark}

We further discuss the characteristic-preserving condition in Section \ref{sec:charpres}. Given the above definitions, we can now state our main theorem. 

\begin{mainthm}[Neukirch--Uchida theorem for 3-manifolds]\label{thm:main}
Let $\mathcal{L} \subset S^3$ be a stably Chebotarev link. Let $h_i: M_i \to S^3 \in \Cov(S^3, \mathcal{L})$, $i = 1, 2$, be two coverings of $S^3$ branched over a finite sublink $L_i \subset \mathcal{L}$. Then:
\begin{enumerate}
    \item Any isomorphism $\sigma: \Gal(M_1, \mathcal{L}_{M_1}) \to \Gal(M_2, \mathcal{L}_{M_2})$ induces a bijection $\sigma_{\mathcal L}: \mathcal{L}_{M_1} \rightarrow \mathcal{L}_{M_2}$.
    \item For any isomorphism $\sigma$ of absolute Galois groups which preserves characteristic, there exists a unique $\alpha\in\Gal(S^3,\mathcal{L})$ with $\sigma(u) = \alpha u \alpha^{-1}$ for all $u \in \Gal(M_1, \mathcal L_{M_1})$; in particular, $\alpha$ induces an isomorphism of covers $M_1 \to M_2$. By Remark \ref{rem:inducedIsom}, two branched covers over $S^3$ are isomorphic if and only if there exists a characteristic-preserving isomorphism between their absolute Galois groups. 
\end{enumerate} 
\end{mainthm}

\begin{remark}
In part (1), the bijection is given by the fact that the image of the decomposition group of $\bar{K}_1 \subset \mathcal L_{\overline{M_1}}$ under $\sigma$ is the decomposition group of $\bar K_2 \subset \mathcal L_{\overline{M_2}}$. See Section \ref{sec:abs} for a definition of decomposition groups. This is a \emph{local correspondence}, which can be found explicitly in Theorem \ref{thm:second}.
\end{remark}

\begin{remark}
Theorem \ref{thm:main}(2), together with Remark \ref{rem:inducedIsom}, gives a bijection between characteristic-preserving isomorphisms $\sigma: U_1 \to U_2$ and elements $\alpha \in \Gal(S^3, \mathcal L)$ with $\alpha U_1\alpha^{-1} = U_2$, via $\sigma = \alpha^*$. In particular, every characteristic-preserving automorphism of $\Gal(S^3, \mathcal{L})$ is inner.
\end{remark}

The proof of Theorem \ref{thm:main} uses local-global principles for profinite group cohomology. The degree one results play the roles of the Hasse principle and the Grunwald--Wang theorem; the degree two results play the role of the Hasse principle for Brauer groups. Throughout, fix a stably Chebotarev link $\mathcal L \subset S^3$ and $h: M \to S^3 \in \Cov(S^3, \mathcal L)$.

\begin{mainthm}\label{thm:maininjective}
Suppose we have a sublink $\mathcal{L}' \subset \mathcal{L}_M$ which contains all but finitely many knot components. Let $p$ be any prime number. Then the natural restriction map
\begin{equation}
    \phi^1: H^1(\Gal(M, \mathcal{L}_M), \mathbb{F}_p) \to \prod_{K \in \mathcal{L}'}H^1(D_{\overline{K}|K}, \mathbb{F}_p)
\end{equation}
is injective.
\end{mainthm}

\begin{mainthm}\label{thm:maininjective2}
The restriction map
\begin{equation}
    \varphi^2: H^2(\Gal(M, \mathcal L_M), \mathbb F_p) \to \prod_{K \subset \mathcal L_M}H^2(D_{\overline K|K}, \mathbb F_p)
\end{equation}
is injective, and its image lies in the direct sum $\bigoplus_{K \subset \mathcal L_M}H^2(D_{\overline K|K}, \mathbb F_p)$. 
\end{mainthm}

\begin{mainthm}\label{thm:mainsurjective}
Suppose we have a finite sublink $L \subset \mathcal{L}_M$. Let $p$ be any prime number. Then the natural restriction map
\begin{equation}
   \psi^1: H^1(\Gal(M, \mathcal{L}_M), \mathbb{F}_p) \twoheadrightarrow \prod_{K \in L}H^1(D_{\overline{K}|K}, \mathbb{F}_p)
\end{equation}
is surjective.
\end{mainthm}

Our main results highlight and motivate further properties of countable links which would further strengthen the connection to prime numbers. We expect these properties to motivate new directions of research in arithmetic topology. In particular, we ask:

\begin{question}
Does there exist a stably Chebotarev link $\mathcal{L} \subset S^3$ which satisfies the following properties?
\begin{itemize}
    \item Every isomorphism between open finite-index subgroups of $\Gal(S^3,\mathcal{L})$ preserves characteristic. 
    \item Any open subgroup $H$ of $\Gal(S^3,\mathcal{L})$ satisfies a Poitou--Tate duality for any finite abelian $H$-module $A$. 
\end{itemize}
\end{question}

A more detailed discussion can be found in Section \ref{sec:future}.

\subsection{Relation to other works}

\subsubsection{Arithmetic Topology}

Versions of the id\`ele class group have been studied (from the cohomological and homological points of view, respectively) in \cite{mihara} and in \cite{uekiniibo, niibo} to establish a class field theory for 3-manifolds. In number-theoretic setups of local-global principles and Poitou--Tate duality, the id\`ele class group is a fundamental object. We expect the frameworks of these works to interact naturally with the notions introduced and questions raised in this paper. 

\medskip

Another related direction in arithmetic topology is the analogue of Iwasawa theory for 3-manifolds \cite{hmm, uekiiwasawa, uekikida}. In our paper, we consider general pro-covers, while Iwasawa theory considers specific towers of $\mathbb{Z}/p^r\mathbb{Z}$-covers (a $\mathbb{Z}_p$ pro-cover) over certain links. 

\subsubsection{Profinite rigidity}\label{sec:profinite}

A group is said to be \emph{profinitely rigid} within a class of residually  finite groups if it is determined up to isomorphism by its profinite completion. This is a central topic for fundamental groups of low-dimensional manifolds and orbifolds; see \cite{reid} for a survey. Since the Galois groups developed in this paper can be expressed as an inverse limit of profinite completions of 3-manifolds with toral boundary, the results in this paper have a natural relationship with profinite rigidity of 3-manifold groups. In this paper, we provide a profinite group invariant which can distinguish 3-manifolds as branched covers of $S^3$. The invariant is not the profinite completion of a single group; it keeps track of ramification along a countable collection of knots, and it is this extra structure that the theorem uses. 

\subsubsection{Anabelian geometry}

Anabelian geometry studies to what extent the geometry and arithmetic of a space can be recovered from its Galois theory and its \'etale fundamental group. The field is based on Grothendieck's anabelian conjectures for curves \cite{grothendieck}. The term anabelian comes from ``beyond abelian" or ``far from abelian", since one needs the fundamental groups to be ``far enough from being abelian" for them to contain enough geometric information. The Neukirch--Uchida theorem serves as the foundational anabelian result dealing with number fields, i.e., arithmetic objects of dimension 0. We wish to highlight a few central theorems in anabelian geometry, and explain what the analogous statements for manifolds should be. For more on anabelian geometry, we refer to the surveys \cite{pop2011lectures}, \cite{mnt}. 

\medskip

We first mention a few anabelian results for number fields that generalise and refine the Neukirch--Uchida theorem. In \cite{ivanov} and \cite{shimizu}, a refinement of the Neukirch--Uchida theorem was given, where sets $S$ with positive Dirichlet density (along with some other mild conditions) are considered as opposed to the set of all primes. In particular, we wish to draw attention to the \emph{characteristic-preserving condition} in \cite{shimizu} (which is referred to as ``good"). Similarly, the anabelian result in \cite{hoshi}, deals with whether homomorphisms between the multiplicative group of fields $K^\times$ arise from field isomorphisms, via a \emph{characteristic-compatible criterion}. These works led us to believe that stably Chebotarev links could satisfy a Neukirch--Uchida theorem and inspired the formulation of our characteristic-preserving condition.

\medskip

We also wish to mention a recent preprint \cite{karshon2026pro} which proves a \emph{tame} Neukirch--Uchida theorem. Since the branched coverings we consider are more analogous to tamely ramified covers, our result is somewhat closer to a \emph{tame} Neukirch--Uchida theorem. Since the preprint is quite recent, the precise analogues of the ideas in \cite{karshon2026pro} are not fully explored in this paper, though it inspired some of the discussions in Section \ref{sec:future}.

\medskip

When considering global fields in characteristic $p$, one has a similar result to Theorem \ref{thm:NU} in Uchida \cite{uchida1977isomorphisms}. Uchida's approach was tremendously improved in the work of Tamagawa \cite{tamagawa}, where an anabelian result for affine curve $X$ over finite fields was given. This was done by identifying the decomposition groups of the points of $X$, but now from the much simpler group $\pi_1^{\et}(X)$. For such curves one has the following exact sequence:
\[1 \to \pi_1^{\et}(X_{\overline{k}})\to \pi_1^{\et}(X) \to G_k \to 1.\]
Under the arithmetic topology dictionary, such curves correspond to hyperbolic mapping tori (i.e., surface bundles over a circle). An analogous result to that of Tamagawa for hyperbolic mapping tori will only require information from decomposition groups of a finite collection of knots in $\mathcal{L}$. This would provide a profinite rigidity result for $\pi_1(M\setminus L)$.

\medskip

Finally, in this paper our local objects come from looking at $\pi_1(\partial V_K)$, these are analogues only to the tamely ramified part of local Galois groups. For a local $p$-adic field $F$, one has the result of Mochizuki \cite{mochizuki1997version}, that there is a bijection
\[\mathrm{Aut}_{\mathbb{Q}_{p}}(F)\longleftrightarrow \mathrm{Out}_{\mathrm{Filt}}(G_{F}).\]
Note that the above requires a special condition on the automorphisms to preserve a group filtration (the ramification filtration). The full group $\text{Out}(G_F)$ is still not well understood. In \cite{gropper2023surfaces}, an arithmetic topology picture that includes the wild ramification in the local setup was developed, and the group $\text{Out}(G_F)$ was related to the mapping class group. Incorporating such a picture into the global setup will allow for more complicated decomposition groups and potentially deeper rigidity results.

\subsection{Summary of analogies of the paper}

We provide the following table for the reader's convenience, summarising the analogies proposed and used in this paper.

{\centering \small
\renewcommand{\arraystretch}{1}
\begin{longtable}{@{} l l l @{}}

\toprule
\textbf{Number theory} & \textbf{3-Manifold topology} & \textbf{Reference} \\
\midrule
Set of all primes in $\mathbb{Z}$ & Stably Chebotarev link $\mathcal{L} \subset S^3$ & Def. \ref{def:chebo}\\
$\operatorname{Spec}(\mathcal{O}_F)$ & Compact connected 3-manifold $M$ & Sec. \ref{sec:background} \\
Prime ideal $\mathfrak{p}$ & Knot $K \in \mathcal{L}_M$ & \S\ref{sec:background} \\
Algebraic closure $\overline{F}$ & Universal $\mathcal{L}_M$-branched cover $\overline{M_{\mathcal{L}_M}}$ & Def. \ref{def:universal} \\
Absolute Galois group $\operatorname{Gal}(\overline{F}/F)$ & $\operatorname{Gal}(M, \mathcal{L}_M) = \varprojlim \widehat{\pi}_1(M \setminus L)$ & Def. \ref{def:galois} \\
Decomposition group $D_{\overline{\mathfrak{p}}|\mathfrak{p}}$ & $D_{\overline{K}|K} \cong \widehat{\mathbb{Z}}^2$ & Def. \ref{def:decomposition} \\
Inertia group $I_{\overline{\mathfrak{p}}|\mathfrak{p}}$ & $I_{\overline{K}|K} = \langle \mu \rangle$ (meridional) & Def. \ref{def:decomposition} \\
Frobenius element $\operatorname{Frob}_{\mathfrak{p}}$ & Longitude class $\lambda$ & Def. \ref{def:frob} \\
Wild inertia $P_{\mathfrak{p}}$ & (Absent) & Sec. \ref{sec:charpres} \\
Tame extension; $D_{\overline{\mathfrak{p}}|\mathfrak{p}}/P_{\mathfrak{p}}\cong \widehat{\mathbb{Z}}^{(p')}\rtimes \widehat{\mathbb{Z}}$ & Branched covers; $D_{\overline{K}|K} \cong \widehat{\mathbb{Z}}^2$ & Sec. \ref{sec:charpres} \\

\midrule
Totally split prime & Totally split knot: $|h^{-1}(K)| = \deg(h)$ & Def. \ref{def:totsplit} \\
Chebotarev density theorem & Chebotarev link property & Def. \ref{def:chebo} \\
Norm $N(\mathfrak{p})$ & $N(K) = e^{\ell(K)}$ & Def. \ref{def:norm} \\
Residue characteristic of $\mathfrak{p}$ & Projection $h(K) \in \mathcal{L} \subset S^3$ & Def. \ref{def:charpres} \\
Dirichlet density & Dirichlet density $\delta(\mathcal{L}', \mathcal{L}_M)$ & Def. \ref{def:dirichlet} \\
\midrule
Hasse principle (in $H^1$) & Injection in $H^1$: Lemma \ref{lma:injective} & Sec. \ref{sec:Local-Global} \\
Grunwald--Wang theorem & Surjection in $H^1$: Lemma \ref{lma:surjective} & Sec \ref{sec:Local-Global} \\
Poitou--Tate duality & Exact sequence via Lefschetz duality: Thm. \ref{thm:exact} & Sec. \ref{sec:Local-Global} \\
Embedding problem (for $\mathbb{F}_p[G]$-extensions) & Lifting property: Lemma \ref{lma:embdding} & Sec. \ref{sec:Local-Global} \\
Local correspondence of primes & $\sigma(D_{\overline{K_1}|K}) = D_{\overline{K_2}|K}$: Thm. \ref{thm:second} & Sec. \ref{sec:localcorr} \\
Neukirch--Uchida theorem & Thm. \ref{thm:main} (characteristic-preserving) & Sec. \ref{sec:proof} \\
\bottomrule
\label{tab:dictionary}
\end{longtable}
}

\subsection{Outline of the paper}

In Section \ref{sec:background}, we present necessary background material for this paper, including facts about the cohomology of profinite groups, Hilbert ramification theory, and Chebotarev links. Furthermore, we present an overview of our proof strategy and draw a parallel to the number-theoretic Neukirch--Uchida theorem. Section \ref{sec:abs} defines and proves basic properties of absolute Galois groups and absolute decomposition groups of 3-manifolds relative to a link. In Section \ref{sec:densities}, we introduce Dirichlet type density and zeta-functions for Chebotarev links, which we use in Section \ref{sec:totsplit} to show that two normal covers with the same set of totally split knots in $S^3$ coincide. In Section \ref{sec:Local-Global} we prove Theorems \ref{thm:maininjective}, \ref{thm:maininjective2}, and \ref{thm:mainsurjective} as well as their consequences; these are used in Section \ref{sec:localcorr} to prove the first part of Theorem \ref{thm:main}. In Section \ref{sec:proof}, we prove the second part of Theorem \ref{thm:main}. Section \ref{sec:future} discusses further questions and directions with regard to links behaving like the set of primes. 

\subsection*{Acknowledgements}
We wish to thank Ted Chinburg, Minhyong Kim, Curtis McMullen, Masanori Morishita, Hiroaki Nakamura, Tony Pantev, Florian Pop, Akio Tamagawa, and Go Yamashita for helpful conversations related to the ideas in this paper. The first author was partially supported by Pantev’s NSF award DMS-2200914 and by Samsung Science and Technology Foundation under Project Number SSTF-BA2201-03. The second author is partially supported by JSPS KAKENHI Grant Number JP23K12969. The third author thanks Ted Chinburg for his mentorship and for introducing the author to arithmetic topology. We also wish to thank the Simons foundation, the RIMS Joint Research Activities and the ICMS for funding travels that further facilitated this collaboration and allowed the authors to have in-person research meetings. 

\section{Preliminaries and Background}  
\label{sec:background}

\subsection{Cohomology of profinite groups}

In this section, we recall some facts about the cohomology of profinite groups. For more about this, see Chapter 1 of \cite{neukirch2013cohomology}. Let $A$ be a Hausdorff, abelian, and locally compact topological group. If $G=\underset{i}{\varprojlim}G_i$ and $A$ is a finite abelian group with trivial $G$-action, then
\begin{equation}\label{eq:limcohom}
    H^n(G, A) \cong \varinjlim_{G_i} H^n(G_i, A).
\end{equation}

\begin{Def}
A group $G$ is \emph{good} if the induced homomorphisms $H^n(\widehat{G}, M) \to H^n(G, M)$ are isomorphisms for any finite $G$-module $M$.
\end{Def}

It is known that 3-manifold groups with toral boundary are cohomologically good; see H.26 in \cite{3groups}.

\begin{Def}
The \emph{Pontryagin dual}, denoted $A^\vee$, is the space of continuous homomorphisms $A \to \mathbb{R}/\mathbb{Z}$.
\end{Def}

If we have an exact sequence 
\begin{equation}
    0\to A\to B\to C\to 0
\end{equation}
then there is a dual exact sequence
\begin{equation}\label{eq:pontDual}
    0\to C^\vee\to B^\vee\to A^\vee\to 0.
\end{equation}

For a finite module $A$, we have functorial isomorphisms 
\begin{equation}
    H_i(G,A)^\vee\cong H^i(G,A^\vee).
\end{equation}
\begin{Def}
Let $(X_i)_{i\in I}$ be a family of abelian topological groups, and $Y_i\subseteq X_i$ a collection of open subgroups. The \emph{restricted product} $\Rprod(X_i,Y_i)$ is the subgroup of $\prod X_i$ for which $x_i\in Y_i$ for almost all $i$ (all but finitely many $i$).
\end{Def}

When $X_i$ are locally compact abelian, and almost all $Y_i$ are compact open, 
\begin{equation}\label{eq:resprodDuality}
    \left(\Rprod(X_i,Y_i)\right)^\vee \cong \Rprod(X_i^\vee,(X_i/Y_i)^\vee).
\end{equation}

\subsection{Hilbert ramification theory}

We now discuss the analogue of Hilbert ramification theory for 3-manifolds from \cite{morishita2012knots} chapter 5. For background on the classical number-theoretic side, see for example Appendix Section 2 of \cite{washington}. 

\medskip

Let $M$ be an oriented connected closed 3-manifold endowed with a base point $b_M$. Let $h: (N, b_N) \to (M, b_M)$ be a finite branched Galois cover with its branch locus a finite link $L$. Let $Y = N \setminus h^{-1}(L)$, $X = M \setminus L$. By abuse of notation, we also write $h: (Y, b_Y) \to (X, b_X)$ for the restriction of $h$ to the exterior of $h^{-1}(L)$. 

\medskip

 Then we have $\Gal(h) = \text{Deck}(h) \cong \pi_1(X, b_X) / h_*(\pi_1(Y, b_Y))$. When $h$ is not Galois, we denote by $\Gal(h)$ the Galois group of the normal closure cover of $h$ (the cover corresponding to the normal core). 

\begin{Def}
Given two covers $h_1:M_1\rightarrow M$ and $h_2:M_2\rightarrow M$ of $M$ branched at finite sublinks of $\mathcal{L}$, there exists a covering containing them both which we call the \emph{compositum of $M_1$ and $M_2$} and denote by $M_{12}$. Let $L_1, L_2$ be the finite branching links in $h_1$ and $h_2$ respectively, $L = L_1 \cup L_2$, and $H_1, H_2$ be the subgroups of $\pi_1(M \setminus L)$ corresponding to the complements of $h_1^{-1}(L)$ and $h_2^{-1}(L)$, both of which are finite covers of $M \setminus L$. Define $\mathcal{M}_{12}$ to be the finite-index cover of $M \setminus L$ corresponding to $H_1 \cap H_2$, and let $M_{12}$ be the Fox completion of the corresponding covering map. (The \emph{Fox completion} of a covering map $Y \to X = M \setminus L$ is the unique extension to a branched cover $N \to M$; see Section 2 of \cite{morishita2012knots}.) 
\end{Def}

\begin{remark}
In the arithmetic topology dictionary, branched covers correspond to number field extensions with ramification, and $M_{12}$ corresponds to the compositum of number fields.
\end{remark}

\begin{prop}
If $M_1\rightarrow M$ is a normal cover, then $M_{12} \to  M_2$ is a normal cover. 
\end{prop} 

\begin{proof}
Let $G=\pi_1(M\setminus L)$.
Since $h_1$ is normal, $H_1$ is normal in $G$. Hence $H_1 \cap H_2$ is normal in $H_2$, so $M_{12} \to M_2$ is a normal covering with deck group $H_2/(H_1 \cap H_2) \cong H_1H_2/H_1$, a subgroup of $\Gal(h_1)$. 
\end{proof}
\begin{Def}
Let $K$ be a knot in $M \setminus L$, and let $K'$ be a component of $h^{-1}(K)$. The \emph{decomposition group} $D_{K'|K}$ of $K'$ over $K$ is defined as 
\begin{equation}
    D_{K'|K} = \{g \in \Gal(h) \mid g(K') = K'\}.
\end{equation}
The \emph{inertia group} $I_{K'|K}$ is defined as 
\begin{equation}
    I_{K'|K} = \ker(D_{K'|K} \twoheadrightarrow \Gal(h|_{K'})).
\end{equation}
\end{Def}

Let $V_K$ be a tubular neighbourhood of $K$. If $\mu$ and $\lambda$ denote the meridian and a chosen longitude of $K$, then we have \[\pi_1(\partial V_K) = \langle \mu, \lambda \mid [\mu, \lambda] \rangle \cong \mathbb{Z}^2.\]
Let $e, f$ denote the branch index and covering degree of $K'$ over $K$, respectively. The decomposition group $D_{K'|K}$ is then an abelian group of order $ef$, and it fits into the exact sequence
\begin{equation}
    1 \to I_{K'|K} \to D_{K'|K} \to \Gal(h|_{K'}) \to 1
\end{equation}
with $I_{K'|K} \cong \mathbb Z / e\mathbb Z$ generated by the meridian, and $\Gal(h|_{K'}) \cong \mathbb Z/f\mathbb Z$ generated by the longitude.

\begin{Def}\label{def:frob}
For an unramified knot $K' \in h^{-1}(K)$ in a finite Galois cover $h: N \to M$ (so that $I_{K'|K} = 1)$, a \emph{Frobenius element} $\Frob_{K'|K} \in D_{K'|K}$ is defined as the image of the longitude $\lambda$ under the quotient $\pi_1(\partial V_K) \to D_{K'|K}$. 
\end{Def}

\subsection{Chebotarev links}\label{sec:chebo}

Chebotarev links were introduced by Mazur \cite{mazur}, when he asked whether there exists an arrangement of a countable link satisfying this property. This was designed to facilitate a topological analogue of the following property of primes in number fields. 

\begin{theorem*}[Chebotarev's density theorem]
Let $L|K$ be a finite Galois extension with group $G$, and let $\mathcal{C}$ be a conjugacy class in $G$. Then 
\begin{equation}
    \lim_{n \to \infty} \frac{|\{\mathfrak{p}\in P_{L|K}(\mathcal{C}):N(\mathfrak{p})\leq n\}|}{|\{\mathfrak{p}\in P_{K}:N(\mathfrak{p})\leq n\}|} = \frac{\#\mathcal{C}}{\#G}
\end{equation}
where $P_K$ is the set of primes of $K$ and $P_{L|K}(\mathcal C)$ is the set of unramified primes with Frobenius in $\mathcal C$.
\end{theorem*}

We have the following definition of the Chebotarev property for a link \cite{mazur}:

\begin{Def}\label{def:chebo}
Let $\mathcal{L} = \bigcup_{i \in \mathbb{Z} > 0}K_i$ be a tame, countably infinite link in a 3-manifold $M$. Let $L_n = \bigcup_{i \leq n}K_i \subset \mathcal L$. Then $\mathcal{L}$ satisfies a \emph{Chebotarev property} if for all $n \in \mathbb{Z}_{\geq 0}$, and any surjective homomorphism $\rho: \pi_1(M \setminus L_n) \to G$ to a finite group $G$, the following density equality holds:
\begin{equation}
    \lim_{\nu \to \infty}\frac{\#\{n < j \leq \nu \mid \rho([K_j]) = C\}}{\nu} = \frac{\#C}{\#G}
\end{equation}
where $[K_j]$ denotes the conjugacy class $K_j \in \pi_1(M \setminus L_n)$.
\end{Def}

Mazur's question was given a positive answer by McMullen in \cite{mcmullen2013knots}, where he showed several examples including:

\begin{Def}\cite{mcmullen2013knots}
Let $M \to S^1$ be the torus bundle with Anosov monodromy $f$ corresponding to $\begin{pmatrix} 2 & 1 \\ 1 & 1 \end{pmatrix}$. The complement of the zero section in $M$ is homeomorphic to the complement of the figure-eight knot. The set of closed periodic orbits of the suspension flow, with the figure-eight knot itself as the first component, is the \emph{planetary link of the figure-eight knot}. 
\end{Def}

We extend Mazur's definition with the following 

\begin{Def}\label{def:stably}
A link $\mathcal{L}$ is \emph{stably Chebotarev} if $\mathcal{L}$ is Chebotarev, and for every finite branched cover $h: N \to M$ branched along a finite sublink of $\mathcal{L}$, the preimage $\mathcal{L}_N$ satisfies the Chebotarev property when ordered by the induced length function
\begin{equation}
    \ell_h(K') = f \cdot \ell(h(K'))
\end{equation}
where $\ell$ is a function $\ell :\{K_i\}\rightarrow\mathbb{R}$ with $\ell(K_i)\geq \ell(K_j)$ if $i\geq j$. 
\end{Def}

\begin{remark}\label{rmk:stably}
This definition of stably Chebotarev is stronger than the one in \cite{ueki2021chebotarev}; we require the Chebotarev property to be satisfied when endowing preimage links with a particular ordering. This enhancement allows us to work compatibly with zeta-functions. The planetary link was shown in \cite{mcmullen2013knots} to be Chebotarev, and shown in \cite{ueki2021chebotarev} to be stably Chebotarev. The orderings in \cite{ueki2021chebotarev} are by length in a generic metric. This applies to our definition. 
\end{remark}

Other properties of countable links, designed to formulate an analogue of idelic class field theory, include \emph{stable genericity} as defined by Mihara \cite{mihara}, and \emph{very admissible} links as defined by Niibo--Ueki \cite{uekiniibo}. We have the following implications: Chebotarev $\Longrightarrow$ stably generic $\Longrightarrow$ very admissible. The first implication was proven by Ueki \cite{ueki2021chebotarev}, and the second was proven by Mihara \cite{mihara}. 

\begin{lemma}\label{lma:chebo}
Let $\mathcal{L} \subset M$ be a Chebotarev link, $L \subset \mathcal{L}$ a finite sublink, $\rho: \pi_1(M \setminus L) \to G$ a surjection onto a finite group $G$, and $C \subset G$ be a conjugacy class. Then there exist infinitely many knots $K \in \mathcal{L} \setminus L$ such that $\rho([K]) = C$.
\end{lemma}

We also apply the Chebotarev property to prove an immediate property about linking numbers between components of a Chebotarev link $\mathcal{L}$. 

\begin{Def}
Let $K, K'$ be disjoint knots in $S^3$. The \emph{linking number} $\lk(K, K') \in \mathbb{Z}$ is the image of $[K] \in H_1(S^3 \setminus K', \mathbb{Z})$ under the canonical isomorphism $H_1(S^3 \setminus K', \mathbb{Z}) \cong \mathbb{Z}$, where the meridian $\mu_{K'}$ is its positive generator. 
\end{Def}

\begin{remark}
Linking numbers on $S^3$ are well-defined; we use the standard identification $H_1(S^3 \setminus L, \mathbb{Z}) \cong \mathbb{Z}^n$ with meridional basis, under which the $i$-th coordinate of $[K]$ is $\lk(K, K_i)$.
\end{remark}

\begin{lemma}\label{lma:linking}
Let $\mathcal{L} \subset S^3$ be a Chebotarev link, $L \subset \mathcal L$ a finite sublink, and $K \in L$. Given any $n \in \mathbb{Z}_{>0}$, $a \in \mathbb{Z}/n\mathbb{Z}$, there exist infinitely many knots $K' \in \mathcal{L} \setminus L$ such that $\lk(K, K') \equiv a$ (mod $n$). In particular, there exist infinitely many $K' \in \mathcal{L}$ such that $\lk(K, K') \neq 0$. 
\end{lemma}

\begin{proof}
Define a homomorphism $\rho: \pi_1(S^3 \setminus L) \to \mathbb{Z}/n\mathbb{Z}$ by 
\begin{equation}
    \rho(\mu_{K'}) = \begin{cases} 1 & K' = K \\ 0 & K' \neq K \end{cases}
\end{equation}
composed with the abelianization map to the first homology. This is surjective because $\rho(\mu_K) = 1$ generates $\mathbb{Z}/n\mathbb{Z}$. For any $K \in L$, $K' \in \mathcal{L} \setminus L$, the image of the free homotopy class $[K']$ under the abelianization has $\mu_K$-component equal to $\lk(K, K')$. Thus, 
\begin{equation}
    \rho([K']) = \lk(K, K') \text{ (mod $n$)}.
\end{equation}
Applying Lemma \ref{lma:chebo} with $G = \mathbb{Z}/n\mathbb{Z}$, $C = a$, and this $\rho$, there exist infinitely many $K' \in \mathcal{L} \setminus L$ with $\rho([K']) = a$, i.e., with $\lk(K, K') \equiv a$ (mod $n$). In particular $\lk(K, K') \neq 0$. 
\end{proof}

\begin{lemma}\label{lma:incompressible}
Let $\mathcal{L} \subset S^3$ be a Chebotarev link with a knot $K \in \mathcal{L}$. Then there exists a knot $K'' \in \mathcal{L}$ such that $\pi_1(\partial V_K)$ is $\pi_1$-injective in $\pi_1(S^3 \setminus L)$ for any finite $L$ containing $K \cup K''$.
\end{lemma}

\begin{proof}
By Lemma \ref{lma:linking}, there exists some $K'' \in \mathcal{L}$ such that $\lk(K, K'') \neq 0$. For any finite $L \subset \mathcal{L}$ containing $K$ and $K''$, any compressing disk $D \subset S^3 \setminus L$ for $\partial V_K$ has $\partial D$ a primitive curve on $\partial V_K$; in $H_1(\partial V_K) = \mathbb{Z}\mu \oplus \mathbb{Z}\lambda$, write $[\partial D] = a\mu + b\lambda$ with $(a, b) \neq (0, 0)$. Now $H_1(S^3 \setminus L)  = \mathbb{Z}^{|L|}$ with basis $\{\mu_{K_i}\}$, and $\lambda_K = \sum_{K_i \in L \setminus K}\lk(K, K_i)\mu_{K_i}$. Substituting, we get
\begin{equation}
    0 = a\mu_K + b\sum_{K_i \neq K}\lk(K, K_i)\mu_{K_i}.
\end{equation}
Linear independence of $\mu_{K_i}$ forces $a = 0$, and $b \cdot \lk(K, K_i) = 0$ for all $K_i$. Since $K'' \in L$ with $\lk(K, K'') \neq 0$, we get $b = 0$. Thus $(a, b) = (0, 0)$, contradicting essentiality of $\partial D$. So we have $\pi_1(\partial V_K) \hookrightarrow \pi_1(S^3 \setminus L)$.
\end{proof}

\begin{corollary}\label{cor:incompressible2}
Let $\mathcal{L} \subset S^3$ be Chebotarev, let $h: M \to S^3$ be a cover in $\Cov(S^3, \mathcal L)$, and let $L$ be a finite sublink of $\mathcal L_M$. Then there is a finite sublink $L'$ containing $L$ such that the exterior of $L'$ in $M$ is aspherical and all its boundary tori are $\pi_1$-injective.
\end{corollary}

\begin{proof}
Choose a nonempty finite sublink $J \subset \mathcal L$ containing $h(L)$ and the branch locus of $h$. By Lemma \ref{lma:chebo}, applied to the mod 2 linking number quotient, there is a component $K_0 \subset \mathcal L \setminus J$ such that $\mathrm{lk}(K, K_0)$ is odd for any $K \subset J$. The link $J \cup K_0$ is thus nonsplit. Using the argument of Lemma \ref{lma:incompressible}, it follows that the boundary tori of $J \cup K_0$ are $\pi_1$-injective. Its exterior is thus irreducible with infinite fundamental group, and hence aspherical by C.1 of \cite{3groups}. Now let $L' = h^{-1}(J \cup K_0)$, which contains $L$. Since $J$ contains the branch locus of $h$, the exterior $M \setminus L'$ is a finite unbranched cover of $S^3 \setminus (J \cup K_0)$. Thus $M \setminus L'$ is also aspherical, and its boundary tori are $\pi_1$-injective, as desired. 
\end{proof}

\subsection{Strategy of proof and the classical Neukirch--Uchida}\label{subsec:NU}

In this paper, we will follow the steps of the proof of the Neukirch--Uchida theorem as they appear in Chapter XII.2 of \cite{neukirch2013cohomology}, \cite{neukirchinvenciones}, \cite{uchida}, providing a parallel picture in the world of $3$-manifolds. To emphasise this parallelism, we sketch the main steps of the proof from \cite{neukirch2013cohomology}. The steps are as follows:
\begin{enumerate}
    \item Two number fields $K_1,K_2$ in $\overline{\mathbb{Q}}$, which are Galois over $\mathbb{Q}$ and have the same set of totally split primes in $\mathbb{Q}$ coincide. This step uses the Chebotarev property to show that the density of totally split primes is $\frac{1}{[K:\mathbb{Q}]}$.
    \item Show that an isomorphism $\sigma$ between absolute Galois groups induces a bijection between primes for which totally split primes in $K_1$ are sent to totally split primes in $K_2$. This is known as the \emph{local correspondence}, which states that a profinite isomorphism induces a bijection between prime ideals in number fields which records the decomposition information. This step is achieved via local-global principles in Galois cohomology, particularly the Hasse principle, the Grunwald--Wang theorem, and Poitou--Tate duality.
    \item Use the first two steps to conclude that any open normal subgroup $N$ of the absolute Galois group contained in the absolute Galois groups of $K_1$ and $K_2$ satisfies $\sigma(N) = N$. Thus, there is an induced isomorphism $\sigma_N$ between finite Galois groups. The local correspondence then implies that $\sigma_N$ sends cyclic subgroups to conjugate cyclic subgroups in $\Gal(\overline{\mathbb{Q}}/\mathbb{Q})/N$.
    \item Given the following diagram of fields
    
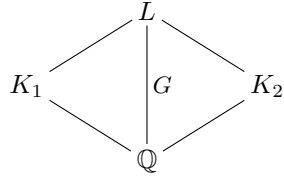
\begin{center}
    \begin{tikzpicture}[every node/.style={inner sep=2pt},]

        \node (L)  at ( 0,  1.8) {$L$};
        \node (K1) at (-1.6, 0.8) {$K_1$};
        \node (K2) at ( 1.6, 0.8) {$K_2$};
        \node (Q) at ( 0, -0.2) {$\mathbb{Q}$};

        \draw (L)  -- node[above left,  font=\small] {} (K1);
        \draw (L)  -- node[above right, font=\small] {} (K2);
        \draw (L)  -- node[right,       font=\small] {$G$} (Q);
        \draw (K1) -- node[below left,  font=\small] {} (Q);
        \draw (K2) -- node[below right, font=\small] {} (Q);
    \end{tikzpicture}
    \end{center}
    one finds an extension $L'$ of $L$, Galois over $\mathbb{Q}$, whose Galois group over $L$ is isomorphic to $\mathbb{F}_p[G]$ (this is a \emph{proper solution to an embedding problem}). Then perform a computation on idempotent elements of $\mathbb{F}_p[G]$ to conclude that $\sigma_N$ acts as conjugation in $\Gal(L'/\mathbb{Q})$, which in turn extends to an inner automorphism of $\Gal(\overline{\mathbb{Q}}/\mathbb{Q})$.
\end{enumerate}
Our analogous proof of the Neukirch--Uchida theorem (Theorem \ref{thm:main}) goes as follows:
\begin{enumerate}
    \item (Sections \ref{sec:densities}, \ref{sec:totsplit}) We show that if $h_i: M_i \to S^3$ are two normal branched covers whose sets of totally split knots in $\mathcal{L}$ completely coincide, then $M_1$ and $M_2$ are the same cover (Corollary \ref{cor:first}). This proof uses \emph{zeta- and $L$-functions} of a Chebotarev link $\mathcal{L} \subset M$. These generalise the functions in Parry-Pollicott \cite{parry1990zeta} and relate different notions of densities of knots, particularly a \emph{Dirichlet density} and a \emph{natural density} (Proposition \ref{prop:densityequiv}). 
    \item (Sections \ref{sec:Local-Global}, \ref{sec:localcorr}) We prove a \emph{local correspondence}, i.e., any isomorphism of Galois groups maps decomposition groups to decomposition groups (Theorem \ref{thm:second}). This is proven using local-global principles in the Galois cohomology of 3-manifolds, which play the roles of the Hasse principle and the Grunwald--Wang theorem in the classical proof of the number-theoretic Neukirch--Uchida theorem (Theorems \ref{thm:maininjective} and \ref{thm:mainsurjective}). To prove Theorem \ref{thm:mainsurjective}, we bypass a full Poitou--Tate duality using Lefschetz duality (Lemma \ref{lma:lefschetz} and Theorem \ref{thm:exact}) to form an exact sequence in $H^1$ which suffices. 
    \item (Section \ref{sec:proof}) Steps 3 and 4 of the number-theoretic proof translate more directly into the topological setting, constructing a common Galois cover of $M_1$ and $M_2$ and using an analogous embedding result (Lemma \ref{lma:embdding}). 
\end{enumerate}

\section{Absolute Galois groups of manifolds}\label{sec:abs}

Let $\mathcal{L}$ be a countable infinite link with an ordering $\{K_i\}_{i\in\mathbb{Z}_{>0}}$ of its knots, and let $\Cov(M, \mathcal{L})$ denote the category of finite branched covers of $M$ branched along a finite sublink of $\mathcal{L}$. By definition, $\Cov(M, \mathcal{L})$ is the filtered colimit of $\Cov(M, L_n)$, the category of finite branched covers of $M$ branched along $L_n=\displaystyle\bigcup_{i=1}^nK_i$. Consequently, $\Cov(M, \mathcal{L})$ forms a Galois category (see \cite{mathew2016galois} Proposition 5.34). 

\begin{Def}
The \emph{absolute Galois group of $M$ relative to $\mathcal{L}$}, denoted $\Gal(M, \mathcal{L})$, is defined to be the fundamental group of the Galois category $\Cov(M, \mathcal{L})$.
\end{Def}

Base points are taken in $M \setminus \mathcal L$, which is path-connected since $\mathcal L$ is a countable union of tame circles.

\medskip

From the general theory of Galois categories, there is a pro-object that pro-represents the fundamental group of $Cov(M,\mathcal{L})$ in the category of profinite covers. We call such a pro-manifold the \emph{universal branched pro-cover}. Throughout the paper we work in a fixed universal branched pro-cover, and denote it by $h_{\mathcal{L}}: \overline{M_\mathcal{L}} \to M$. From its universal properties, this is a minimal object such that any $h \in \Cov(M, \mathcal{L})$ factors through $h_{\mathcal{L}}$.

\medskip

For the foundational theory of Galois categories, we refer the reader to \cite{sga} (Expos\'e V). See \cite{uekiniibo} for an example of a similar philosophy of absolute Galois groups of manifolds relative to families of infinite links.

\begin{Def}\label{def:universal}
The \emph{universal $\mathcal{L}$-branched cover} $h_{\mathcal{L}}: \overline{M_\mathcal{L}} \to M$ is the fixed universal branched pro-cover of $\mathrm{Cov}(M, \mathcal{L})$. 
\end{Def}

\begin{prop}\label{prop:Galoisgroups}
We have three equivalent characterisations of the Galois group:
\begin{enumerate}
    \item $\Gal(h_{\mathcal{L}})$
    \item $\underset{L}{\varprojlim}\widehat{\pi}_1(M \setminus L)$
    \item $\underset{h \in \Cov(M, \mathcal{L})}{\varprojlim}\Gal(h)$
\end{enumerate}
\end{prop}

\begin{proof}
We first show $\Gal(h_{\mathcal{L}}) \cong \Gal(M, \mathcal{L})$. This is due to the definition of a universal pro-cover $h_\mathcal{L}: \overline{M_{\mathcal{L}}} \to M$, being the pro-object that pro-represents $\Gal(M, \mathcal{L})$. 

\medskip

We now show \[\varprojlim_L\widehat{\pi}_1(M \setminus L) \cong \varprojlim_{h \in \Cov(M, \mathcal{L})}\Gal(h).\]
For each finite sublink $L \subset \mathcal{L}$, the covers in $\Cov(M, \mathcal{L})$ branched at $L$ are parametrised by finite-index subgroups of $\pi_1(M \setminus L)$, and their inverse limit of deck transformation groups is precisely the profinite completion of $\widehat{\pi}_1(M \setminus L)$. The full system $\Cov(M, \mathcal{L})$ is therefore organised as a double inverse system: first over covers at each fixed $L$, then over finite sublinks $L \subset \mathcal{L}$ ordered by inclusion. Since both indexing systems are cofiltered, the interchange of limits gives

\[    \varprojlim_{h \in \Cov(M, \mathcal{L})}\Gal(h) \cong \varprojlim_{L \subset \mathcal{L}}\widehat{\pi}_1(M \setminus L),\]
as desired. 

\medskip

Finally, we show \[\varprojlim_L\widehat{\pi}_1(M \setminus L) \cong \Gal(M, \mathcal{L}).\] Since the category $\Cov(M, \mathcal{L})$ is the filtered colimit of $\Cov(M, L_n)$, we get that its fundamental group (i.e., the absolute Galois group of $M$ relative to $\mathcal{L}$) is isomorphic to the inverse limit of the fundamental groups of the $\mathrm{Cov}(M, L_n)$, $\widehat\pi_1(M\setminus L_n)$. 
\end{proof}

Let $N$ be a branched cover in $\Cov(M, \mathcal{L})$, i.e., $h: N \to M$. By the Galois correspondence, the open subgroup $\Gal(N, \mathcal{L}_N)$ corresponds to $h$.

\begin{Def}
Let $K$ be a knot in $\mathcal{L} \subset M$. We denote by $\overline{K}$ a choice of component in the universal $\mathcal{L}$-branched cover lying over $K$.
\end{Def}

The above is equivalent to a choice for each $h \in \Cov(M, \mathcal{L})$ of some component $K_h \in h^{-1}(K)$ such that if $N_1 \underset{h_1}{\to} N_2 \underset{h_2}{\to} M$ is a tower of branched covers, $h_1(K_{h_1}) = K_{h_2}$.

\begin{Def}\label{def:decomposition}
Let $K$ be a knot in $\mathcal{L}$ and $\overline{K}$ a component above it in $\overline{M_\mathcal{L}}$. Define an absolute \emph{decomposition group} to be the decomposition group of $\overline{K}$
\[    D_{\overline{K}|K}=\{g \in \Gal(M,\mathcal{L}) \mid g(\bar{K}) = \bar{K}\}\]
and an absolute \emph{inertia group} of $K$ to be the inertia group of $\overline{K}$ 
\[
    I_{\overline{K}|K}= \ker(D_{\overline{K}|K} \twoheadrightarrow \Gal(\bar{h}|_{\bar{K}}))).
\]
\end{Def}

\begin{remark}
Just as in number theory, if $\overline{K}_1, \overline{K}_2$ are two choices of knots in a universal cover lying over the same knot $K$, their decomposition groups are related by conjugation in the absolute Galois group. 
\end{remark}

Let $h: N \to M$ be a finite Galois cover in $\Cov(M, \mathcal L)$, let $\rho_h: \Gal(M, \mathcal L) \to \Gal(h)$ be the quotient map, and let $\overline K$ be a compatible system over $K$ with $K_h$ its component at level $h$. Then
\begin{equation}
    \rho_h(D_{\overline K|K}) = D_{K_h|K} \ \ \ \ \ \rho_h(I_{\overline K|K}) = I_{K_h|K}
\end{equation}
This follows from the definitions and is used throughout.

\begin{theorem}\label{thm:profinitetorus}
Let $\mathcal{L} \subset S^3$ be a Chebotarev link. For any $K \in \mathcal{L}$ and any lift $K' \in h^{-1}(K)$ in any $h: M \to S^3 \in \Cov(S^3, \mathcal{L})$, we have $ D_{\overline{K'}|K'} \cong\widehat{\pi}_1(\partial V_{K'})\cong \widehat{\mathbb{Z}}^2$.
\end{theorem}

\begin{proof}
We prove $D_{\overline{K}|K} \cong \widehat{\pi}_1(\partial V_{K})$; the case of lifts $K'$ follows since $D_{\overline{K'}|K'} = D_{\overline{K}|K} \cap \Gal(M, \mathcal{L}_M)$ is open in $\widehat{\mathbb{Z}}^2$, hence isomorphic to $\widehat{\mathbb{Z}}^2$. The compatible peripheral maps induce a continuous homomorphism
\begin{equation}
    f: \widehat{\pi_1(\partial V_K)} \cong \widehat{\mathbb Z}^2 \to D_{\bar K \mid K}.
\end{equation}
We prove that $f$ is injective and surjective.

\medskip

\underline{Surjectivity:} Let $h \in \Cov(S^3, \mathcal L)$. The image of the peripheral group is the decomposition group $D_{K_h \mid K}$. Thus, the image of $f$ projects onto every decomposition group at a finite level. Since $f(\widehat{\mathbb Z}^2)$ is compact, it is closed in $D|_{\bar K \mid K} = \varprojlim_hD_{K_h \mid K}$, and dense due to finite-level surjectivity, so the image is $D_{\bar K \mid K}$.

\medskip

\underline{Injectivity:} Given $\phi: \mathbb{Z}^2 \twoheadrightarrow A$ for $A$ finite abelian, set $n = |A|$. Fix any finite $L \subset \mathcal{L}$ containing $K$. By Lemma \ref{lma:linking}, there exists $K'' \in \mathcal{L} \setminus L$ with $\lk(K, K'') \equiv 1$ (mod $n$). Define $\psi: H_1(S^3 \setminus (K \cup K'')) \to A$ by $\psi(\mu_K) = \phi(\mu_K), \psi(\mu_{K''}) = \phi(\lambda_K)$. Since $[\lambda_K] = \lk(K, K'')[\mu_{K''}]$ in $H_1(S^3 \setminus (K \cup K''))$ and $\lk(K, K'') \equiv 1$ (mod $n$), we get $\psi(\lambda_K) = \phi(\lambda_K)$, so $\psi$ extends $\phi$. The Fox completion of the corresponding abelian cover lies in $\Cov(S^3, \mathcal{L})$ and has decomposition group $\psi(\pi_1(\partial V_K)) = \phi(\mathbb{Z}^2) = A$ at $K$. Thus, every finite quotient $\phi: \mathbb{Z}^2 \twoheadrightarrow A$ factors through $f$, i.e. $\widehat\phi: \widehat{\mathbb Z}^2 \xrightarrow{f} D_{\bar K|K} \to A$. Thus, $\ker f \subset \bigcap_\phi \ker\widehat{\phi} = \{1\}$, so $f$ is injective.
\end{proof}

\begin{remark}
For more on the geometric group-theoretic properties and profinite groups in the context of 3-manifolds, see \cite{3groups}.
\end{remark}

An analogue of a totally split prime in number theory is defined as follows. 

\begin{Def}\label{def:totsplit}
Given a finite covering $h: N \to M \in \Cov(M, \mathcal{L})$, of degree $n$, a component $K \in \mathcal{L}$ is \emph{totally split} if $h^{-1}(K)$ has $n$ components. For $K$ unramified in $h$, this holds if and only if $K$ is totally split in the Galois closure of $h$. See Remark \ref{lma:totsplit} for equivalent characterizations.
\end{Def}

\begin{remark}\label{lma:totsplit}
The following equivalent definitions follow immediately from the Galois correspondence and the definition of the absolute decomposition group. Let $\mathcal{L} \subset M$ be a stably Chebotarev link, and take $h: N \to M \in \Cov(M, \mathcal{L})$. Let $K \in \mathcal{L}$. 
\begin{enumerate}
    \item A lift $K' \in \mathcal{L}_N$ has $D_{K'|K} = 1$ if and only if there exists some compatible inverse system $\overline{K}$ over $K$ containing $K'$ for which $D_{\overline{K}|K} \subset \Gal(N, \mathcal{L}_N)$.
    \item $K$ is totally split in $N$ if and only if for all compatible inverse systems $\overline{K}$ over $K$, $D_{\overline{K}|K} \subset \Gal(N, \mathcal{L}_N)$.
    \item Let $\rho: \Gal(M, \mathcal{L}) \to \Gal(h)$ be the quotient map. Let $L$ be the ramified set of the covering $h$. Then $K \in \mathcal{L} \setminus L$ is totally split if and only if $\rho([K]) = 1$.
\end{enumerate} 
\end{remark}

An analogue of a totally indecomposable prime in number theory is defined as follows. 

\begin{Def}
Given a finite Galois covering $h: N \to M \in \Cov(M, \mathcal{L})$, a knot $K \in \mathcal{L}$ is \emph{totally indecomposable} in $h$ if there is a single component $K'$ in $h^{-1}(K)$. 
Let $\overline{h}: \overline{M_{\mathcal{L}}} \to \widetilde{M}$ be a pro-covering. A knot $K \in \mathcal{L}_{\widetilde M}$ is \emph{totally indecomposable} in $\overline{h}$ if for every finite intermediate cover $h': M' \to \widetilde{M}$ factoring through $\overline{h}$, the knot $K$ has a unique preimage in $(h')^{-1}(K)$.
\end{Def}

\begin{remark}
The following equivalent definitions of total indecomposability follow from the Galois correspondence.
\begin{enumerate}
    \item For any components $K'$ in $h^{-1}(K)$, $D_{K'|K} = \Gal(h)$. This definition will apply for a profinite cover. 
    \item For any compatible inverse system $\overline{K}$ above $K$, $D_{\overline{K}|K} \cdot \Gal(N, \mathcal{L}_N) = \Gal(M, \mathcal{L})$. 
    \item For pro-coverings, $\Gal(\overline{h}) \subset D_{\overline{K}|K}$. 
\end{enumerate}
\end{remark}

\section{Dirichlet and natural densities}\label{sec:densities}

Recall the following definition:
\begin{Def}
The \emph{Dirichlet density} of a set of primes $P$ in a number field is 
\begin{equation}
    \delta(P) = \lim_{s \to 1^+}\frac{\sum_{\mathfrak{p} \in P}N(\mathfrak{p})^{-s}}{\sum_{\mathfrak{p}}N(\mathfrak{p})^{-s}}
\end{equation}
where the sum in the denominator runs over all primes, and $N(\mathfrak{p})$ is the absolute norm. 
\end{Def}

Many number-theoretic computations of prime densities are stated and computed in terms of \emph{Dirichlet densities}. However, the above definition of Chebotarev links is a direct analogue of the \emph{natural density}.

\medskip

In order to make certain analytic computations more accessible, one desires a notion of Dirichlet density. The first step is to assign a notion of size to the knots in a Chebotarev link. In order for this size function to be meaningful and relate back to the original notion of density we used, we require that our size function maintains the ordering of the knot complements in $\mathcal{L}$.

\begin{example}
The following are examples of Chebotarev links with appropriate size functions.
\begin{itemize}
    \item Theorem 1.1 of \cite{mcmullen2013knots} states that the closed orbits of the geodesic flow in the unit tangent bundle of a closed hyperbolic surface form a Chebotarev link. The length function here is the length of the associated geodesics, and here the length function determines the ordering.
    \item Theorem 1.2 of \cite{mcmullen2013knots} proves that the same holds for closed orbits of any topologically mixing pseudo-Anosov flow on a closed 3-manifold. This construction holds for any fibered hyperbolic knot complement.
    \item Let $p_n$ be the $n$-th prime number. Given any Chebotarev link $\mathcal{L} \subset M$, define $\ell:\{K_i\} \rightarrow \mathbb{R}$, by $\ell(K_i) = \ln(p_i)$. This length function also preserves the order of knots in $\mathcal{L}$.
\end{itemize}
\end{example}

\begin{remark}
Since the natural density does not depend on ``size" but just on the order, one can take any size function such that the order is preserved. For geometric length functions arising from the geodesic flow, such as the first two above, the analogue of $\psi(x) \sim x$ is the prime geodesic theorem, which is available by \cite{parry1990zeta}. We use the prime-number length throughout this paper because it is defined for any Chebotarev link without requiring a geometric structure. No result in this paper depends on the choice.
\end{remark}

For a size function $\ell$ that preserves the order of knots in $\mathcal L$ and satisfies $\ell(K_j) \to \infty$, the set $\{j \mid N(K_j) \leq x\}$ is an initial segment $\{1, \dots, \nu(x)\}$ of the ordering, with $\nu(x) < \infty$, and $\nu(x) \to \infty$ as $x \to \infty$. Hence, for any conjugacy class $C$ in a finite group $G$ and any surjection $\rho: \pi_1(M \setminus L_n) \to G$,
\begin{equation}
    \lim_{x \to \infty}\frac{\#\{n < j \mid N(K_j) \leq x, \rho([K_j]) = C\}}{\#\{n < j \mid N(K_j) \leq x\}} = \lim_{\nu \to \infty}\frac{\#\{n < j \leq \nu \mid \rho([K_j]) = C\}}{\nu}
\end{equation}
whenever the right-hand side exists. The left-hand side is the density counted by norm, which is what the zeta functions below see.

\begin{Def}\label{def:norm}
Given an order-preserving size function $\ell$ on a Chebotarev link, define the \emph{norm} $N(K_i) = e^{\ell(K_i)}$. Define $\pi_{\mathcal{L}}(x):=\underset{N(K_j)\leq x}{\sum}1$ and $\pi_{\mathcal{L},C}(x):=\underset{\underset{\rho([K_j])=C}{N(K_j)\leq x}}{\sum}1$. These are our ``prime counting functions". 
\end{Def}

In summary, a Chebotarev link $\mathcal{L} \subset M$ has an associated length function. If we are given a finite covering branched over finitely many components of $\mathcal{L}$, we will also associate a length function to the inverse image of the link and show that under this length function, one also has $L$-functions with desirable properties. For the rest of this paper, we will use the prime number length described above.

\begin{Def}
Suppose we have a Chebotarev link $\mathcal{L} \subset S^3$. Given a knot $K \in \mathcal{L} \subset S^3$ and a component $K'$ in $h^{-1}(K)$, the \emph{degree} of $K'$ over $K$ is defined as the degree of induced covering $h|_{K'}: K' \to K$. If $K$ is unramified in $h$, this is equal to the size of the decomposition group $D_{K'|K}$.
\end{Def}

\begin{Def}
Suppose we have a Chebotarev link $\mathcal{L} \subset S^3$. We use the size function $\ell: \{K_i \in \mathcal{L}\} \rightarrow \mathbb{R}$, defined by $\ell(K_i) = \ln(p_i)$. Let $h: M \rightarrow S^3 \in \Cov(S^3, \mathcal{L})$. For a connected component $K' \in h^{-1}(K)$ with degree $n$ over $K$, we define a new length function $\ell_h: \{K' \in \mathcal{L}_M\} \rightarrow \mathbb{R}$ by $\ell_h(K') = n\ell(K)$. For brevity we will just use $\ell$ when the context is clear.
\end{Def}

\begin{Def}\label{def:dirichlet}
The \emph{Dirichlet density} of a sublink $L$ in $\mathcal{L}_M$ is defined to be 
\begin{equation}
    \delta(L, \mathcal{L}_M) = \lim_{s \to 1^+} \frac{\sum_{K \in L}\frac{1}{N(K)^s}}{\sum_{K \in \mathcal{L}_M}\frac{1}{N(K)^s}}.
\end{equation}
The \emph{zeta function} of $\mathcal{L}_M$ is defined to be 
\begin{equation}
    \zeta(\mathcal{L}_M, s) = \prod_{K_i \in \mathcal{L}_M}\frac{1}{1 - N(K_i)^{-s}}.
\end{equation}
Given a finite sublink $L_n \subset \mathcal{L}_M$, a conjugacy class in a finite group $C \subset G$ and a surjection $\rho: \pi_1(M - L_n) \to G$, define the \emph{$\zeta$-function relative to $(C, G, \rho)$} to be
\begin{equation}
    \zeta(\mathcal{L}_M, C, s) = \prod_{K_i \in \mathcal{L}_M \setminus L_n, \rho([K_i]) = C}\frac{1}{1 - N(K_i)^{-s}}.
\end{equation}
\end{Def}

The existence of this limit is not assumed a priori; Proposition \ref{prop:densityequiv} proves it exists under the Chebotarev hypothesis. The following is a consequence for Laplace--Stieltjes transforms and is a classical Abelian theorem for Mellin transforms; see Chapter V of \cite{widder}.

\begin{lemma}\label{lma:calc}
If $f, g: [1, \infty) \to [0, \infty)$ are nondecreasing with $g(x) \to \infty$, and $f(x)/g(x) \to \alpha$, then for $F(s) = s\int_1^\infty f(x)x^{-s-1}dx$ and $G(s) = s\int_1^\infty g(x)x^{-s-1}dx$ (both finite for $\text{Re}(s) > 1$), if one also has that $G(s)\to \infty$ as $s \to 1^+$, then $F(s)/G(s) \to \alpha$ as $s \to 1^+$. 
\end{lemma}

In our application $g = \psi$ is the function \eqref{eq:psi} below. The counting estimate in the proof of Proposition \ref{prop:densityequiv} gives $\psi(x) \sim x$, hence $G(s) \to \infty$ as $s \to 1^+$.

\begin{prop}\label{prop:densityequiv}
Fix a Chebotarev link $\mathcal{L} \subset S^3$, a covering $h: M \to S^3 \in \Cov(S^3, \mathcal{L})$, a conjugacy class $C$ in a finite group $G$, a finite sublink $L_n \subset \mathcal{L}_M$ consisting of the first $n$ knots, and a surjection $\rho: \pi_1(M \setminus L_n) \to G$. If 
\begin{equation}
    \lim_{x \to \infty}\frac{\pi_{\mathcal{L}_M, C}(x)}{\pi_{\mathcal{L}_M}(x)} = \lim_{\nu \to \infty}\frac{\#\{n < j \leq \nu \mid \rho([K_j]) = C\}}{\nu} = \frac{\#C}{\#G}
\end{equation}
then 
\begin{equation}
    \delta(\{K_i \in \mathcal{L}_M \mid \rho([K_i]) = C\}, \mathcal{L}_M) = \frac{\#C}{\#G}.
\end{equation}
Thus, if the Chebotarev property holds for the natural density, then it holds for Dirichlet density.
\end{prop}

\begin{proof}
The idea is to relate the Dirichlet density to ratios of logarithmic derivatives of zeta functions via a Mellin-transform Abelian theorem, then use the Chebotarev hypothesis to evaluate these ratios. Define 
\begin{equation}\label{eq:psi}
    \psi(x) = \sum_{n=1}^\infty\sum_{e^{n\ell(K_i)}\leq x}\ell(K_i)
\end{equation}
and
\begin{equation}
    \psi_C(x):=\sum_{n=1}^\infty\sum_{e^{n\ell(K_i)}\leq x, \rho([K_i]) = C}\ell(K_i).
\end{equation}
Recall that we are using $\ell(K_i) = \ln(p_i)$, where $p_i$ is the $i$-th prime. We have the standard equivalences $\pi_C/\pi \sim \theta_C/\theta \sim \psi_C/\psi$; we note that they rely on the prime number theory applied to the norms $N(K_i) = p_i$, which is available because our size function assigns to each knot the logarithm of the corresponding prime. Here $\pi_{\mathrm{spl}}$ counts knots totally split in the normal closure of $h$ and has positive density by Lemma \ref{lma:densitytotallysplit}. So we have:
\begin{equation}\label{eq:39}
   \lim_{x\rightarrow\infty}\frac{\pi_{\mathcal{L}_M,C}(x)}{\pi_{\mathcal{L}_M}(x)}= \lim_{x\rightarrow\infty}\frac{\psi_C(x)}{\psi(x)}.
\end{equation}
Now we have
\begin{equation}\label{eqt:logzeta}
    \ln(\zeta(\mathcal{L}_M, s)) = \sum_i-\ln(1-N(K_i)^{-s}) = \sum_n\sum_i\frac{1}{ne^{\ell(K_i)ns}}
\end{equation}
The terms with $n \geq 2$ are bounded by $\sum_i N(K_i)^{-2} \leq \deg(h)\sum_pp^{-2}$ uniformly in $s > 1$, so $\ln\zeta(\mathcal L_M, s) = \sum_iN(K_i)^{-s} + O(1)$, likewise for $\zeta(\mathcal L_M, C, s)$, and $\sum_iN(K_i)^{-s} \to \infty$ as $s \to 1^+$ since $\pi_{\mathcal L_M}(x) \gg x/\log(x)$. So the logarithmic derivative is computed as follows. For $\text{Re}(s) > 1$, 
\begin{equation}
    -\frac{\zeta '(\mathcal{L}_M, s)}{\zeta(\mathcal{L}_M,s)} = \sum_i\frac{e^{-s\ell(K_i)}\ell(K_i)}{1-e^{-s\ell(K_i)}}=\sum_n\sum_i\frac{\ell(K_i)}{e^{ns\ell(K_i)}}=s\int^\infty_1\frac{\psi(x)}{x^{s+1}}dx.    
\end{equation}
Here the last equality is by partial summation. Performing a similar calculation restricted to a class $C$ yields
\begin{equation}
    -\frac{\zeta'(\mathcal{L}_M, C, s)}{\zeta(\mathcal{L}_M, C, s)} = s\int_1^\infty\frac{\psi_C(x)}{x^{s+1}}dx.
\end{equation}
Because $\lim_{x \to \infty}\frac{\psi_C(x)}{\psi(x)} = \frac{\#C}{\#G}$, Lemma \ref{lma:calc} dictates that the ratio of the logarithmic derivatives approaches the same limit:
\begin{equation}
    \lim_{s \to 1^+}\frac{-\frac{\zeta '(\mathcal{L}_M,C,s)}{\zeta(\mathcal{L}_M,C,s)}}{-\frac{\zeta '(\mathcal{L}_M,s)}{\zeta(\mathcal{L}_M,s)}}=\frac{\#C}{\#G}.
\end{equation}
By Equations \ref{eq:39} and \ref{eqt:logzeta}, the Dirichlet density is given by $\lim_{s \to 1^+}\frac{\ln\zeta(\mathcal{L}_M, C, s)}{\ln\zeta(\mathcal{L}_M, s)}$. Because $\ln\zeta(\mathcal{L}_M, s)$ approaches infinity as $s \to 1^+$, applying L'Hopital's rule (in the form requiring only the denominator to diverge) to this ratio of logarithms yields exactly the limit of the logarithmic derivatives above. Thus, the Dirichlet density is $\frac{\#C}{\#G}$, as desired.
\end{proof}

\section{Totally split knots}\label{sec:totsplit}

In this section, we seek to prove the analogous statement to step 1 discussed in subsection \ref{subsec:NU} for branched covers of 3-manifolds. For this section, fix a stably Chebotarev link $\mathcal{L} \subset S^3$. Given $h: M \to S^3 \in \Cov(S^3, \mathcal{L})$ we seek to show that the homeomorphism type of $M$ is determined by the set of knots in $\mathcal{L}$ which are totally split under $h$. 

\begin{lemma}\label{lma:densitytotallysplit}
For a normal covering $h: M \to S^3 \in \Cov(S^3, \mathcal{L})$, the natural density of knots in $\mathcal{L}$ which are totally split with respect to $h$ is $\frac{1}{|\Gal(h)|}$.
\end{lemma}

\begin{proof}
Let $L \subset \mathcal{L}$ be the finite set of ramified knots in $h$, and let $\rho: \Gal(S^3, \mathcal{L}) \to \Gal(h)$ be the natural quotient map. By Remark \ref{lma:totsplit}, $K \in \mathcal{L} \setminus L$ is totally split if and only if $\rho([K])$ is trivial. Since $\mathcal{L}$ is a Chebotarev link, the natural density of knots $K \in \mathcal{L}$ such that $\rho([K])$ is trivial, i.e., $K$ is totally split, is exactly $\frac{1}{|\Gal(h)|}$. 
\end{proof}

\begin{remark}\label{rmk:stable}
Since $\mathcal{L}$ is stably Chebotarev, this lemma is also true for general normal covers $N \rightarrow M$. 
\end{remark}

\begin{lemma}\label{lma:density}
Let $h: N \to M \in \Cov(M, \mathcal{L}_M)$. Then the Dirichlet density of degree-1 knots in $\mathcal{L}_N$ is 1. 
\end{lemma}

\begin{proof}
We will show that knots of degree greater than one have density 0. Let $\mathbb{L}_h = \{K \in \mathcal{L}_N \mid \deg(K) > 1\}$. Let $n = \deg(h)$. Note that each knot in $\mathcal{L}_M$ has at most $n$ knot components above it in $\mathcal{L}_N$, and so
\begin{equation}
   \sum_{K \in \mathbb{L}_{h}}\frac{1}{N(K)^s} \leq n\sum_{K \in h(\mathbb{L}_{h})}\frac{1}{N(K)^{2s}}.
\end{equation}
Here, we use that for any $K \in \mathcal{L}_N$ with $\deg(K) = d \geq 2$ over $h(K)$, the definition of the size function gives $\ell(K) = d\ell(h(K))$, so $N(K) = e^{\ell(K)} = N(h(K))^d \geq N(h(K))^2$. Thus, $N(K)^{-s} \leq N(h(K))^{-2s}$, and since each knot in $\mathcal{L}_M$ has at most $n$ components above it in $\mathcal{L}_N$, the bound follows by summing over the image $h(\mathbb{L}_{h})$. It follows that 
\begin{equation}\label{eq:dirichlet}
     \delta(\mathbb{L}_{h}, \mathcal{L}_N) \leq \limsup_{s \rightarrow 1^+}\frac{n\sum_{K \in h(\mathbb{L}_{h})}N(K)^{-2s}}{\sum_{K \in \mathcal{L}_N}N(K)^{-s}}.
\end{equation}
By Lemma \ref{lma:densitytotallysplit} and Proposition \ref{prop:densityequiv}, the Dirichlet density of totally split knots in $\mathcal{L}_M$ is $1/|\Gal(h)| > 0$. Each totally split knot $K \in \mathcal{L}_M$ contributes $n$ degree-1 lifts to $\mathcal{L}_N$, so the degree-1 contribution to the denominator is therefore at least $n\sum_{D_{K'|K}=1}N(K)^{-s}$. Thus, the denominator surpasses a positive proportion of $\sum p_i^{-s}$, which diverges as $s \to 1^+$. The sum $\sum_{K \in h(\mathbb{L}_{h})}N(K)^{-2s}$ is bounded by $\sum_{K \in \mathcal{L}_M}N(K)^{-2s} \leq (\deg(M \to S^3))\sum_ip_i^{-2s}$, which converges for $s > 1/2$. Thus, in Equation \ref{eq:dirichlet}, the numerator converges while the denominator diverges, so the Dirichlet density is zero.
\end{proof}

\begin{theorem}\label{thm:subcover}
Suppose $h_1: M_1 \to S^3$ is a finite normal branched covering, and $h_2: M_2 \to S^3$ is a finite branched covering over a stably Chebotarev link $\mathcal{L} \subset S^3$. Suppose that, for all but finitely many components $K \subset \mathcal L$, if some component of $h_2^{-1}(K)$ has degree 1 over $K$, then $K$ is totally split in $h_1$. Then $M_1$ is a subcover of $M_2$. 
\end{theorem}

\begin{proof}
Let $M_{12}$ be the compositum of $M_1$ and $M_2$ over $S^3$ and let $L_{12}$ be the union of the ramification loci. We will show that almost all knots in $\mathcal{L}_{M_2}$ totally split in $M_{12} \to M_2$, which will imply our theorem by Proposition \ref{prop:densityequiv}. 

\medskip

Let $K_2 \in h_2^{-1}(K)$, be a knot of degree 1 over $K \notin L_{12}$ outside the finite exceptional set of the hypothesis. Let $K_{12}$ be a knot above it in $M_{12}$. Since for any composite cover and a knot in it over $K$ one has that $f_{K_{12}|K}\leq f_{K_1|K}f_{K_2|K}$, the assumption of the theorem gives us that $f_{K_{12}|K}=1$. On the other hand $f_{K_{12}|K}=f_{K_{12}|K_2}f_{K_{2}|K}$ and so $f_{K_{12}|K_2}=1$. Since $M_{12} \to M_2$ is normal, so $K_2$ totally splits in $M_{12} \to M_2$. By Lemma \ref{lma:density}, the degree-1 knots in $\mathcal{L}_{M_2}$ have Dirichlet density 1. So a Dirichlet density 1 of the knots in $\mathcal{L}_{M_2}$ totally split in $M_{12}$, so $M_{12} = M_2$ by Lemma \ref{lma:densitytotallysplit}, Remark \ref{rmk:stable} and Proposition \ref{prop:densityequiv} applied to $M_{12} \to M_2$. 
\end{proof}

From the above we have the following corollaries:

\begin{corollary}\label{cor:normal}
Let $h: M \to S^3 \in \Cov(S^3, \mathcal{L})$. If, for all but finitely many knots, every knot with a degree-1 lift in $M$ is totally split in $M$, then $h$ is a normal covering.
\end{corollary}

The following is a topological analogue of the classical result that a finite Galois extension of a number field is determined by its set of completely split primes.

\begin{corollary}\label{cor:first}
If $h_1: M_1 \to S^3, h_2: M_2 \to S^3$ are two normal branched covers in $\Cov(S^3, \mathcal{L})$ whose sets of totally split knots in $\mathcal{L}$ completely coincide, then $M_1$ and $M_2$ are the same cover.
\end{corollary}

\section{Local-global principles for links in 3-manifolds}\label{sec:Local-Global}

In this section, we develop local-global principles for absolute Galois groups and decomposition groups. In particular, we prove Theorems \ref{thm:maininjective}, \ref{thm:maininjective2}, and \ref{thm:mainsurjective}, cohomological results which will be crucial for our proof of a local correspondence between profinite decomposition groups, akin to the corresponding results in the proof of the Neukirch--Uchida theorem. Notably, we bypass the need for a Poitou--Tate type result by utilizing exact sequences for 3-manifolds. For the remainder of this section, fix a stably Chebotarev link $\mathcal{L} \subset S^3$ and a covering $h: M \to S^3 \in \Cov(S^3, \mathcal{L})$. Let $p$ be a prime number. Throughout this section $A$ is a finite self-dual module with trivial action.

\begin{Def}
The \emph{local unramified cohomology group} at a knot $K$ is defined as the image of the \emph{inflation map}:
\begin{equation}
    H^i_{ur}(D_{\overline{K}|K}, A) = \im\left(H^i(D_{\overline{K}|K}/I_{\overline{K}|K}, A^{I_{\overline{K}|K}}) \to H^i(D_{\overline{K}|K}, A)\right)
\end{equation}
 and $A^{I_{\overline{K}|K}}$ is the submodule of $A$ fixed by the action of $I_{\overline{K}|K}$. The \emph{global unramified cohomology group} is the subgroup of global classes whose restrictions are locally unramified at every knot in $\mathcal{L}$:
\begin{equation}
    H^i_{ur}(M, \mathcal{L}, A) = \ker\left(H^i(\Gal(M, \mathcal{L}_M), A) \to \prod_{K \in \mathcal{L}_M}H^i(I_{\overline{K}|K}, A)\right).
\end{equation}
\end{Def}

Let $h: M \to S^3 \in \Cov(S^3, \mathcal{L})$, and define
\begin{equation}
    P^i(\Gal(M, \mathcal{L}_M), A) = \Rprod_{K \in \mathcal{L}_M}H^i(D_{\overline{K}|K}, A)
\end{equation}
where the above product is restricted with respect to $H^i_{ur}(D_{\overline{K}|K}, A)$.

\medskip

Our central results are an injection
\begin{equation}
    \phi^i: H^i(\Gal(M, \mathcal{L}_M), \mathbb{F}_p) \to \prod_{K \in \mathcal{L}'}H^i(D_{\overline{K}|K}, \mathbb{F}_p)
\end{equation}
for $i = 1, 2$, and some $\mathcal{L}' \subset \mathcal{L}_M$ containing all but finitely many knot components, and a surjection
\begin{equation}
   \psi^1: H^1(\Gal(M, \mathcal{L}_M), \mathbb{F}_p) \twoheadrightarrow \prod_{K \in L}H^1(D_{\overline{K}|K}, \mathbb{F}_p)
\end{equation}
where $L$ contains finitely many components.

\begin{lemma}\label{lma:injective}[Theorem \ref{thm:maininjective}]
Let $h: M \to S^3 \in \Cov(S^3, \mathcal{L})$. Suppose we have a subset $\mathcal{L}' \subset \mathcal{L}_M$ which contains all but finitely many knot components. Let $p$ be a prime number. Then the natural restriction map
\begin{equation}
    \phi^1: H^1(\Gal(M, \mathcal{L}_M), \mathbb{F}_p) \to \prod_{K \in \mathcal{L}'}H^1(D_{\overline{K}|K}, \mathbb{F}_p)
\end{equation}
is injective.
\end{lemma}

\begin{proof}
Take some nontrivial $\varphi: \Gal(M, \mathcal{L}_M) \to \mathbb{F}_p$ in $\ker(\phi^1)$. Then take the covering $h_\varphi: M_\varphi \to M$ corresponding to $\ker(\varphi)$, a subgroup of $\Gal(M, \mathcal{L}_M)$. Since $\varphi$ is nontrivial, this is a normal degree-$p$ covering. However, since $\varphi \in \ker(\phi^1)$, it follows that $\varphi|_{D_{\overline{K}|K}}$ is trivial for all $K \in \mathcal{L}'$. Note that if $K' \in h_\varphi^{-1}(K)$, $D_{K'|K} \cong \varphi(D_{\overline{K}|K}) = 1$. In other words, all but finitely many knots in $\mathcal{L}_M$ are totally split in the degree-$p$ cover $M_\varphi \to M$. By Lemma \ref{lma:densitytotallysplit} and Remark $\ref{rmk:stable}$, it follows that $h_\varphi$ is a trivial covering, contradicting the fact that $h_\varphi$ has degree $p$. There cannot be a nontrivial $\varphi \in \ker(\phi^1)$, so $\phi^1$ is injective. 
\end{proof}

\begin{lemma}\label{lma:injh2finite}
Let $h: M \to S^3 \in \mathrm{Cov}(S^3, \mathcal L)$ be a finite branched cover. Then the restriction map
\begin{equation}
    \varphi^2: H^2(\Gal(M, \mathcal L_M), \mathbb F_p) \to \prod_{K \subset \mathcal L_M}H^2(D_{\overline K|K}, \mathbb F_p)
\end{equation}
is injective, and its image lies in the direct sum $\bigoplus_{K \subset \mathcal L_M}H^2(D_{\overline K|K}, \mathbb F_p)$. 
\end{lemma}

\begin{proof}
By Corollary \ref{cor:incompressible2}, Proposition \ref{prop:Galoisgroups}, the continuity of cohomology \cite{serre}, and goodness, 
\begin{equation}
    H^2(\Gal(M, \mathcal L_M),\mathbb F_p) \cong \varinjlim_LH^2(\pi_1(M\setminus L); \mathbb F_p)
\end{equation}
where $L$ ranges over the cofinal family of finite sublinks whose exteriors are aspherical and have $\pi_1$-injective boundary tori. Let $L$ be such a link with exterior $X_L$, and neighborhoods $V_K$ chosen to be disjoint. By Corollary 6.20 of \cite{wilkes} (with the remark following its Definition 6.11), $D_{\overline K|K} \to \widehat{\pi_1}(X_L)$ is injective for $K \subset L$, so $H^2(D_{\overline K|K}; \mathbb F_p) \cong H^2(\partial V_K; \mathbb F_p)$ compatibly with restriction, and $\varphi^2(x)$ vanishes on $L$ if and only if $\varphi_L^2(x_L) = 0$. 

\medskip

Suppose $x \in H^2(\Gal(M, \mathcal L_M); \mathbb F_p)$ is represented by $x_L \in H^2(X_L; \mathbb F_p)$. The restriction of $x_L$ through any $K$ not in $L$ factors through $H^2(\widehat{\pi_1}(V_K); \mathbb F_p) = H^2(\widehat{\mathbb Z}; \mathbb F_p) = 0$, for a tubular neighborhood $V_K \subset M \setminus L$. Thus $\varphi^2(x)$ is supported on a finite link $L$, and so lies in the direct sum $\bigoplus_{K \subset \mathcal L_M}H^2(D_{\overline K|K}, \mathbb F_p)$. 

\medskip

Now suppose $\varphi^2(x) = 0$. Since solid tori have vanishing second cohomology, the Mayer-Vietoris exact sequence for $M=X_L\coprod_{\partial V_K} V_K$ gives 
\begin{equation}\label{eq:MV}
    \dots \to \bigoplus_{K \subset L}H^1(\partial V_K; \mathbb F_p) \to H^2(M; \mathbb F_p) \xrightarrow{\iota_L^*}H^2(X_L; \mathbb F_p) \xrightarrow{\varphi_L^2} \bigoplus_{K \subset L}H^2(\partial V_K; \mathbb F_p) \to H^3(M; \mathbb F_p) \to \dots
\end{equation}
Hence $x_L = \iota_L^*\eta$ for some $\eta \in H^2(M; \mathbb F_p)$. Since $\mathcal L_M$ is Chebotarev, we may apply Lemma \ref{lma:chebo} to $\pi_1(M \setminus L) \twoheadrightarrow H_1(M; \mathbb F_p)$, we may choose components $K_1, \dots, K_r \subset \mathcal L_M \setminus L$ whose classes span $H_1(M; \mathbb F_p)$. Choose $L' \supset L \cup K_1 \cup \dots \cup K_r$ (or some extension to a link whose complement is aspherical and with $\pi_1$-injective torus boundaries, from Corollary \ref{cor:incompressible2}). By the exact sequence \ref{eq:MV}, the kernel of $\iota^*_{L'}$ is equal to the image of \[\bigoplus_{K \subset L'}H^1(\partial V_K; \mathbb F_p) \to H^2(M; \mathbb F_p).\] Taking Poincar\'e duality of dimension 2 on the left side, and of dimension 3 on the right side, we conclude that \[\ker(\iota^*_{L'})\cong \im (\bigoplus_{K \subset L'}H_1(\partial V_K; \mathbb F_p) \to H_1(M; \mathbb F_p)).\] From our construction of $L'$ we get that the kernel is $H^2(M; \mathbb F_p)$, and thus $\iota_{L'}^* = 0$. Therefore, $\iota_{L'}^*\eta = 0$, but also the image of $x$ at the level $L'$ has $x_{L'}=\iota_{L'}^*\eta$, and so $x = 0$. 
\end{proof}

Given the above proof for injectivity of the $\phi^2$ for finite branched covers, we show an injectivity statement for any closed subgroup of $\Gal (S^3,\mathcal{L})$. The proof given below is essentially the same as the one given in \cite{neukirch1968algebraische} Satz II. We give it here for completeness.

\begin{corollary}\label{cor:geninjective}
Let $\mathcal{L} \subset S^3$ be a stably Chebotarev link, $G = \Gal(S^3, \mathcal{L})$, $H \subset G$ a closed subgroup with corresponding cover $h_H$, and $\mathcal{L}_H = h_H^{-1}(\mathcal{L})$. Then the natural map
\begin{equation}
    \phi^2: H^2(H, \mathbb{F}_p) \rightarrow \prod_{J \in \mathcal{L}_H}H^2(D_{\overline{J}|J}, \mathbb{F}_p)
\end{equation}
is injective.
\end{corollary}

\begin{proof}
Since $H$ is the inverse limit of all open subgroups of $G$ containing it, we have by Equation \ref{eq:limcohom}:
\begin{equation}
    H^2(H, \mathbb{F}_p) = \varinjlim_{H \subset U \subset_o G}H^2(U, \mathbb{F}_p).
\end{equation}
Similarly, given such an inverse system of open subgroups $U_i \subset G$ containing $H$, and a knot $J\in \mathcal{L}_H$, we have that $D_{\bar J|J}$ is the inverse limit of $D_{\bar{J}|J_{U_i}}$ where $J_{U_i}=h_{H|U}(J)$, the image of $J$ under the branched cover $h_{H|U_i}: M_H \to M_U$. 
 For any $H \subset U$, naturality of restriction in group cohomology gives a commutative square
\begin{equation}\label{eq:diagramh2}
\begin{tikzcd}[row sep=3em, column sep=6em]
H^2(U,\mathbb{F}_p)
  \arrow[d, "\phi^2_{J_{U_i}}"]
  \arrow[r, "\mathrm{Res}_i"]
&
H^2(H,\mathbb{F}_p)
  \arrow[d,  "\phi^2_{J}"]
\\
 H^2(D_{\overline{J}|J_{U_i}}, \,\mathbb{F}_p)
  \arrow[r,"\mathrm{Res}_{J_{i}}"]
&
H^2(D_{\overline{J}|J},\,\mathbb{F}_p)
\end{tikzcd}
\end{equation}
In the above diagram, we denote by $\phi^2_J$ the $J$-th component of the map $\phi^2$. If $x \in \ker(\phi^2_H)$ then $x = \Res_i(x_i)$ for $x_i \in H^2(U_i, \mathbb F_p)$ for large enough $i$, and since it is in the kernel of $\phi_H^2$, Diagram \ref{eq:diagramh2} gives
\begin{equation}
    \Res_{J_i}(\phi^2_{J_{U_i}}(x_i))=\phi^2_J(\Res_i(x_i))=\phi^2_J(x)=1.
\end{equation}
For each $i$, let $L_i$ denote the set of knots $K \in \mathcal L_{U_i}$ with $\phi^2_K(x_i)\neq 1$; these $L_i$ form a projective system. Since $U_i$ are open, by the second part of Lemma \ref{lma:injh2finite}, $L_i$ are finite for each $i$. On the other hand, suppose $J\in\underset{\leftarrow}{\lim}L_i$. Since $\Res_{J_i}(\phi^2_{J_{U_i}}(x_i))=1$ for all $i$, but also for all large $i$ we would have $\phi^2_{J_{U_i}}(x_i)\neq1$ (by \ref{eq:limcohom}, $H^2(D_{\overline J|J}; \mathbb F_p) = \varinjlim_iH^2(D_{\overline J|J_{U_i}}; \mathbb F_p)$ and the $\phi^2_{J_{U_i}}(x_i)$ are compatible, so $\phi_J^2(x) = 1$ forces $\phi^2_{J_{U_i}}(x_i) = 1$ for some $i$), leading to a contradiction. Thus $\underset{\leftarrow}{\lim}L_i=\emptyset$, and since $L_i$ are all finite, for large enough $i$, all $L_i=\emptyset$. For such large enough $i$, $x_i\in\ker(\phi^2_{U_i})$. By the first part of Lemma \ref{lma:injh2finite}, $x_i=1$, and so also $x=1$.
\end{proof}

\begin{lemma}\label{lma:localduality}
We have the following isomorphism, coming from the local duality for decomposition groups:
\begin{equation}
    \Xi^1: P^1(\Gal(M, \mathcal{L}_M), A) \to P^1(\Gal(M, \mathcal{L}_M), A)^\vee.
\end{equation}
\end{lemma}

\begin{proof}
The decomposition group $D_{\overline{K}|K} \cong \widehat{\mathbb{Z}}^2$ is a Poincar\'e group of dimension 2 at every prime $p$ (see for example \cite{serre} section 4.5). By the duality isomorphism for Poincar\'e groups (3.4.6 of \cite{neukirch2013cohomology}), the cup product pairing $H^1(D_{\overline{K}|K}, A) \times H^1(D_{\overline{K}|K}, A) \to H^2(D_{\overline{K}|K}, A) \cong A$ is nondegenerate. To establish the self-duality of the restricted product $P^1$, we must show that $H^1_{ur}$ is its own orthogonal complement under the Poincar\'e duality pairing. Let $D_{\overline{K}|K} = \langle \mu, \ell \rangle$ with inertia subgroup $I_{\overline{K}|K} = \langle \mu \rangle$. By definition, $\chi$ is an element of the unramified cohomology group if and only if $\chi(\mu) = 0$. For $\chi_1, \chi_2 \in H^1_{ur}$, we have
\begin{equation}
    \chi_1 \cup \chi_2(\mu \wedge \ell) = \chi_1(\mu)\chi_2(\ell) - \chi_1(\ell)\chi_2(\mu) = 0.
\end{equation}
Thus, $H^1_{ur} \subset (H^1_{ur})^\perp$. Furthermore, we have $|H^1| = |A|^2$, $|H^1_{ur}| = |A|$, and by Poincar\'e duality we have $|(H^1_{ur})^\perp| = |A|^2/|A| = |A|$. So $H^1_{ur} = (H^1_{ur})^\perp$. Thus, by Equation \ref{eq:resprodDuality} we are done.
\end{proof}

To establish a surjective map, we require the following theorem regarding finite modules with trivial actions, which serves as a step toward a full Poitou--Tate duality. 

\medskip

For a finite sublink $L \subset \mathcal L_M$, we write $\Gal(M, L) = \widehat\pi_1(M \setminus L)$; this is the maximal quotient of $\Gal(M, \mathcal L_M)$ unramified outside $L$, i.e. the quotient by the closed normal subgroup generated by the $I_{\overline K|K}$, $K \notin L$.

\begin{lemma}\label{lma:lefschetz}
Let $X = M \setminus \bigcup_{K \in L}V_K^{\circ}$ be a compact oriented aspherical 3-manifold with boundary $\partial X = \bigcup_{K \in L}\partial V_K$ a finite disjoint union of $\pi_1$-injective tori. Then the sequence
\begin{equation}\label{eq:58}
    H^1(\Gal(M, L); A) \to \prod_{K \in L}H^1(D_{\overline{K}|K}; A) \to H^1(\Gal(M, L); A)^\vee
\end{equation}
is exact. The first map is given by $\loc_L$, the second map is given by $(\loc_L)^\vee \circ \Xi^1$, and $\Xi^1$ is the Poincar\'e duality pairing on $\partial X$.
\end{lemma}

\begin{proof}
Since $X$ is a compact oriented 3-manifold with boundary, Lefschetz duality gives an isomorphism $\Phi: H^2(X, \partial X; A) \to H^1(X; A)^\vee$ defined by $\Phi(\beta)(\alpha) = \langle \alpha \cup \beta, [X, \partial X] \rangle$, using $A \cong A^\vee$. The long exact sequence of the pair $(X, \partial X)$ gives exactness of
\begin{equation}\label{eq:61}
    H^1(X; A) \xrightarrow{\iota^*} H^1(\partial X; A) \xrightarrow{\delta} H^2(X, \partial X; A).
\end{equation}
We claim that $\Phi \circ \delta = (\iota^*)^\vee \circ \Xi^1$. Indeed, for $\alpha \in H^1(X; A)$ and $\beta \in H^1(\partial X, A)$
\begin{equation}
    \Phi(\delta\beta)(\alpha) = \Xi^1(\beta, \iota^*\alpha) = \langle \alpha \cup \delta\beta, [X, \partial X] \rangle = \langle \iota^*\alpha \cup \beta, [\partial X] \rangle
\end{equation}
where the middle equality is the projection formula relating the connecting homomorphism $\delta$ to the restriction $\iota^*$, and the last equality holds because $\partial X = \bigcup_K\partial V_K$ decomposes into the local Poincar\'e duality pairing on each torus. By exactness of \ref{eq:61}, $\ker((\iota^*)^\vee \circ \Xi^1) = \ker(\Phi \circ \delta) = \ker(\delta) = \im(\iota^*)$. To pass to profinite cohomology, we invoke Theorem 6.21 of \cite{wilkes}, which establishes that $(\pi_1(X), \{\pi_1(\partial V_K)\}_{K \in L})$ is a good pair (Definition 6.1 of Wilkes \cite{wilkes}). The comparison from topological to profinite cohomology is a map of long exact sequences (see the diagram preceding Definition 6.1 of Wilkes \cite{wilkes}), so it intertwines $\iota^*$ with $\loc_L$ and identifies $H^1(\partial V_K; A)$ with $H^1(D_{\overline{K}|K}; A)$ compatibly with these maps. The topological exactness established above therefore gives $\ker((\loc_L)^\vee \circ \Xi^1) = \im(\loc_L)$, which is exactness of \ref{eq:58}.
\end{proof}

We use this lemma to prove the exact sequence from which a local-global principle for surjectivity will follow.

\begin{theorem}\label{thm:exact}
Let $h: M \to S^3 \in \Cov(S^3, \mathcal{L})$. The sequence
\begin{equation}
    H^1(\Gal(M, \mathcal{L}_M), A) \to P^1(\Gal(M, \mathcal{L}_M), A) \to H^1(\Gal(M, \mathcal{L}_M), A)^\vee
\end{equation}
is exact. Denote the first map by $\loc_{\mathcal{L}}$; the second map is given by $(\loc_\mathcal{L})^\vee \circ \Xi^1$.
\end{theorem}

\begin{proof}
By Corollary \ref{cor:incompressible2}, for any finite sublink $L$ we can add components so that all boundary components are $\pi_1$-injective and aspherical; Lemma \ref{lma:lefschetz} then gives the desired exact sequence. We now pass to the full restricted product $P^1 = P^1(\Gal(M, \mathcal{L}_M), A)$ over all of $\mathcal{L}_M$. The exactness of this sequence is equivalent to showing that $\im(\loc_{\mathcal{L}})$ equals its own orthogonal complement in $P^1$ under the pairing
\begin{equation}
    P^1 \times P^1 \rightarrow \bigoplus H^2(D_{\overline{K}|K}, A) \to \mathbb{Q}/\mathbb{Z}
\end{equation}
where the first map is the componentwise cup product and the second map is summation.

\medskip

\underline{Claim 1: $\im(\loc_{\mathcal{L}}) \subset \im(\loc_{\mathcal{L}})^\perp$.}

\medskip

Given two classes $x, y \in H^1(\Gal(M, \mathcal{L}_M), A)$, the local cup product $(\loc_{\mathcal{L}}(x))_K \cup (\loc_{\mathcal{L}}(y))_K \in H^2(D_{\overline{K}|K}, A)$ can be nontrivial only if at least one of $x_{D_{\overline{K}|K}}$ or $y_{D_{\overline{K}|K}}$ lies in the ramified part of $H^1(D_{\overline{K}|K}, A)$ (follows from self-orthogonality of $H^1_{ur}$, i.e., Lemma \ref{lma:localduality}). Since $x$ and $y$ are globally unramified outside a finite set of knots, this can only occur for finitely many $K$. Let $L \subset \mathcal{L}_M$ be a finite subset including all such $K$, chosen to have aspherical complement and $\pi_1$-injective boundary by Corollary \ref{cor:incompressible2}. Because $x$ and $y$ are unramified outside $L$, there exist classes $\overline{x}, \overline{y} \in H^1(\Gal(M, L), A)$ such that $x = \inf(\overline{x})$ and $y = \text{inf}(\overline{y})$. Thus, the global cup product evaluates as:
\begin{equation}
    \loc_\mathcal{L}(x) \cup \loc_\mathcal{L}(y) = \pi_L(\loc_{\mathcal{L}}(x)) \cup \pi_L(\loc_{\mathcal{L}}(y)) = \loc_L(\overline{x}) \cup \loc_L(\overline{y}) = 0
\end{equation}
where the last equality uses Lemma \ref{lma:lefschetz}, giving $\im(\loc_{\mathcal{L}}) \subset \im(\loc_{\mathcal{L}})^\perp$.

\medskip

\underline{Claim 2: $\im(\loc_{\mathcal{L}})$ is closed in $P^1$.}

\medskip

We now show that $\loc_{\mathcal{L}}$ is proper, i.e., the preimage of any compact subset of $P^1$ is compact in $H^1(\Gal(M, \mathcal{L}_M), A)$. A compact neighbourhood of the identity in $P^1$ is contained in 
\begin{equation}
    P_L = \prod_{K \in L}H^1(D_{\overline{K}|K}, A) \times \prod_{K \in \mathcal{L}_M \setminus L}H^1_{ur}(D_{\overline{K}|K}, A)
\end{equation}
for some finite $L$. Since $H^1(\Gal(M, \mathcal{L}_M), A)$ is discrete, it suffices to show that $\loc_{\mathcal{L}}^{-1}(P_L)$ is finite. Any $x \in \loc_{\mathcal{L}}^{-1}(P_L)$ satisfies $x|_{I_{\overline{K}|K}} = 0$ for all $K \notin L$, hence factors through the maximal quotient of $\Gal(M, \mathcal{L}_M)$ unramified outside $L$, which is $\Gal(M, L) = \widehat{\pi}_1(M \setminus L)$. Via the inflation-restriction exact sequence:
\begin{equation}
    0 \to H^1(\Gal(M, L), A) \to_{\text{inf}} H^1(\Gal(M, \mathcal{L}_M), A)
\end{equation}
we have $\loc_{\mathcal{L}}^{-1}(P_L) \subset \im(\text{inf})$ (particularly using that inflation is injective in $H^1$ for normal subgroups). Since the exterior of $L$ is a compact 3-manifold with boundary, $\pi_1(M \setminus L)$ is finitely presented, hence $\Gal(M, L)$ is topologically finitely generated, making $H^1(\Gal(M, L), A)$ finite. Therefore, $\loc_{\mathcal{L}}^{-1}(P_L)$ is finite, and $\loc_{\mathcal{L}}$ is proper. A proper map from a discrete space into a locally compact space has closed image, so $\im(\loc_{\mathcal{L}})$ is closed in $P^1$. 

\medskip

\underline{Claim 3: $\im(\loc_{\mathcal{L}})^\perp \subset \im(\loc_{\mathcal{L}})$.}

\medskip

Let $x \in P^1$ satisfy $x \cup \loc_{\mathcal{L}}(y) = 0$ for all $y \in H^1(\Gal(M, \mathcal{L}_M), A)$. For any finite $L$ containing the finite set of knots where $x$ is ramified and whose exterior is aspherical with $\pi_1$-injective boundary, projecting to $P^1(\mathcal L)$ gives $|pi_L(x) \in \im(\loc_L)^\perp = \im(\loc_L)$ by Lemma \ref{lma:lefschetz}. Here we use that every class in $H^1(\Gal(M, \mathcal L), A)$ is in the image of the inflation $H^1(\Gal(M, \mathcal L_M), A)$, so orthogonality against all global classes at the infinite level implies orthogonality against all classes at the finite level.
Since $\im(\loc_{\mathcal{L}})$ is closed by Claim 2:
\begin{equation}
    x \in \bigcap_L\pi_L^{-1}(\im(\loc_L)) \subset \bigcap_L\pi_L^{-1}(\pi_L(\im(\loc_{\mathcal{L}}))) = \overline{\im(\loc_{\mathcal{L}})} = \im(\loc_{\mathcal{L}}).
\end{equation}
The first inclusion uses that $\im(\loc_L)$ at the finite level equals $\pi_L(\im(\loc_{\mathcal{L}}))$, i.e., every finite-level element in $\im(\loc_L)$ lifts to a global element in $\im(\loc_{\mathcal{L}})$. The third equality uses Claim 2.

\medskip

Claim 1 and 3 together give $\im(\loc_{\mathcal{L}}) = \im(\loc_{\mathcal{L}})^\perp$, which under the identification via $\Xi^1$ is equivalent to exactness of the stated sequence.
\end{proof}

\begin{lemma}\label{lma:surjective}[Theorem \ref{thm:mainsurjective}]
Let $h: M \to S^3 \in \Cov(S^3, \mathcal{L})$. Suppose we have a finite set of knots $L \subset \mathcal{L}_M$. Then the natural restriction map
\begin{equation}
   \psi^1: H^1(\Gal(M, \mathcal{L}_M), \mathbb{F}_p) \twoheadrightarrow \prod_{K \in L}H^1(D_{\overline{K}|K}, \mathbb{F}_p)
\end{equation}
is surjective.
\end{lemma}

\begin{proof}
By Theorem \ref{thm:exact}, and since the restriction map on all of $\mathcal{L}_M$ is injective by Lemma \ref{lma:injective}, we have the exact diagram 
\[
\begin{tikzcd}
    & 1 
        \arrow[r]
        \arrow[d]
    & \displaystyle\Rprod_{K \in \mathcal{L}_M \setminus L} H^{1}(D_{\overline{K}|K}, \mathbb{F}_p)
        \arrow[r, "\text{id}"]
        \arrow[d]
    & \displaystyle\Rprod_{K \in \mathcal{L}_M \setminus L} H^{1}(D_{\overline{K}|K}, \mathbb{F}_p)
        \arrow[r]
        \arrow[d, "f"]
    & 1 \\
1 
    \arrow[r]
& H^{1}(\Gal(M, \mathcal{L}_M), \mathbb{F}_p) 
    \arrow[r]
& \displaystyle P^1(\Gal(M, \mathcal{L}_M), \mathbb{F}_p)
    \arrow[r, "g"]
& H^{1}(\Gal(M, \mathcal{L}_M), \mathbb{F}_p)^\vee
\end{tikzcd}
\]
Here, $g = (\loc_{\mathcal{L}})^\vee \circ \Xi^1$ as in Theorem \ref{thm:exact}, and $f$ is a restriction of $g$. Because the quotient of the full restricted product by the product over $K \in \mathcal{L}_M \setminus L$ is isomorphic to the finite direct product over $L$, the snake lemma yields the exact sequence 
\begin{equation}\label{eq:snake}
    H^1(\Gal(M, \mathcal{L}_M), \mathbb{F}_p) \to \prod_{K \in L}H^1(D_{\overline{K}|K}, \mathbb{F}_p) \to \coker(f).
\end{equation}
By the definition of $f$ and Pontryagin duality (Equation \ref{eq:pontDual}), we have 
\begin{equation}
    \coker(f) \cong \ker\left(H^1(\Gal(M, \mathcal{L}_M), \mathbb{F}_p) \to \left(\Rprod_{\mathcal{L}_M \setminus L}H^1(D_{\overline{K}|K}, \mathbb{F}_p)\right)^\vee\right)^\vee.
\end{equation}
However, by Lemmas \ref{lma:injective} and \ref{lma:localduality}, this kernel is trivial, and so the cokernel of $f$ is also trivial. Thus, Equation \ref{eq:snake} gives the required surjection.
\end{proof}

Writing $H = \bigcap_i U_i$ with $U_i$ open and $i$ large enough that the finitely many $K \in L$ have distinct images in $M_{U_i}$, Lemma \ref{lma:surjective} for $M_{U_i}$ together with the direct limits $H^1(H; \mathbb F_p) = \varinjlim H^1(U_i; \mathbb F_p)$, $H^1(D_{\overline K|K} \cap H; \mathbb F_p) = \varinjlim_iH^1(D_{\overline K|K} \cap U_i; \mathbb F_p)$ of \eqref{eq:limcohom} gives:

\begin{corollary}\label{cor:surjgeneral}
Let $\mathcal{L} \subset S^3$ be a stably Chebotarev link, and let $G = \Gal(S^3, \mathcal{L})$. Let $H \subset G$ be any closed subgroup with corresponding cover $h_H$, and let $L = \{K_1, \dots, K_r\} \subset h_H^{-1}(\mathcal{L})$ be a finite set. Then the restriction map
\begin{equation}
    \psi^1: H^1(H, \mathbb{F}_p) \to \prod_{K \in L}H^1(D_{\overline{K}|K} \cap H, \mathbb{F}_p)
\end{equation}
is surjective, where $D_{\overline{K}|K}$ is the decomposition group of $\overline{K}$ in the $H$ pro-cover. 
\end{corollary}

We finish this section with a lifting property. For this we need the following corollary of Lemma \ref{lma:surjective}:

\begin{corollary}\label{cor:nsw924}
Let $\mathcal{L} \subset S^3$ be a stably Chebotarev link and $h: M \to S^3 \in \Cov(S^3, \mathcal{L})$ a normal cover with $L \subset \mathcal{L}$ a finite set of components. Let $G = \Gal(h)$. Then the natural restriction map
\begin{equation}
    \psi^1: H^1(\Gal(S^3, \mathcal{L}), \mathbb{F}_p[G]) \to \prod_{K \in L} H^1(D_{\overline{K}|K}, \mathbb{F}_p[G])
\end{equation}
is surjective.
\end{corollary}

The corollary above follows directly Lemma \ref{lma:surjective} by applying Shapiro's lemma to the left side, and applying Shapiro's lemma together with Mackey's double coset formula (see \cite{neukirch2013cohomology} Proposition 1.5.11) for computing the restriction of the induced module on the right hand side.

\medskip

The above allows us to show the following lifting property:

\begin{lemma}\label{lma:embdding}
Let $h: N \to S^3 \in \Cov(S^3, \mathcal{L})$ be a normal finite cover and take a prime number $p$. Let $G = \Gal(h)$. Given an exact sequence of finite groups as in the bottom row of the below diagram, where $\Gal(h)$ acts on $\mathbb F_p[G]$ by left multiplication, there exists a surjective homomorphism $\psi: \Gal(S^3, \mathcal{L}) \to E$ such that 
\[
\begin{tikzcd}[column sep=large, row sep=large]
& & & \mathrm{Gal}(S^3, \mathcal{L}) \arrow[d] \arrow[dl, two heads, "\psi"'] \\[1em]
1 \arrow[r] 
& \mathbb{F}_p[G] \arrow[r] 
& E \arrow[r] 
& \Gal(h) \arrow[r]
& 1
\end{tikzcd}
\]
commutes. 
\end{lemma}
In light of the theorems proved in this section, the proof of the above lemma follows in the same way as Proposition 9.2.9 in \cite{neukirch2013cohomology}, which we bring here for completeness.
\begin{proof}
Let $G = \Gal(h)$. By Shapiro's lemma (1.6.4 of \cite{neukirch2013cohomology}), $H^2(\Gal(h), \mathbb{F}_p[G]) \cong H^2(1, \mathbb{F}_p) = 0$. By Proposition 3.5.9 in \cite{neukirch2013cohomology}, there exists some $\psi_0: \Gal(S^3, \mathcal{L}) \to E$ such that 
\[
\begin{tikzcd}[column sep=large, row sep=large]
& & & \mathrm{Gal}(S^3, \mathcal{L}) \arrow[d] \arrow[dl, "\psi_0"'] \\[1em]
1 \arrow[r] 
& \mathbb{F}_p[G] \arrow[r] 
& E \arrow[r] 
& \Gal(h) \arrow[r]
& 1
\end{tikzcd}
\]
commutes, but $\psi_0$ is not necessarily surjective. Since the above diagram is commutative, given a knot $K$ which is totally split with respect to $h$, any homomorphism $D_{\overline{K}} \to E$ lifting the canonical projection to $\Gal(h)$ must land entirely in the kernel $\mathbb{F}_p[G]$. By Lemma \ref{lma:densitytotallysplit}, we can pick a finite collection of knots $\{K_i\}_{i=1}^r \subset \mathcal{L}$ which are unramified and totally split with respect to $h$. We then choose local homomorphisms $\varphi_i: D_{\overline{K_i}} \to \mathbb{F}_p[G]$ whose images fully generate $\mathbb{F}_p[G]$. Let
\begin{equation}
    \psi^1: H^1(\Gal(S^3, \mathcal{L}), \mathbb{F}_p[G]) \to \prod_{i=1}^rH^1(D_{\overline{K_i}}, \mathbb{F}_p[G])
\end{equation}
be the restriction map, which is surjective by Corollary \ref{cor:nsw924}. This surjection implies that there exists a 1-cocycle $x \in Z^1(\Gal(S^3, \mathcal{L}), \mathbb{F}_p[G])$ representing a class in $H^1$ such that its local restriction is $x|_{D_{\overline{K_i}}}(g) = \varphi_i(g) \cdot (\psi_0(g))^{-1}$ for $g \in D_{\overline{K_i}}$. Let $\psi: \Gal(S^3, \mathcal{L}) \to E$ be defined as $\psi(g) = x(g)\psi_0(g)$; since the $K_i$ are totally split, the canonical projection of $D_{\overline{K_i}}$ into $\Gal(h)$ is trivial. Since $\psi_0$ makes the diagram commute, the image $\psi_0(D_{\overline{K_i}})$ must lie entirely in the kernel $\mathbb{F}_p[G]$. Thus, $\psi|_{D_{\overline{K_i}}} = \varphi_i$. Since the images of $\varphi_i$ generate $\mathbb{F}_p[G]$, it follows that the image of $\psi$ contains $\mathbb{F}_p[G]$, which is the kernel of the map $E \to \Gal(h)$. However, $\psi$ makes the above diagram commute; since its image contains the kernel, the image of $\psi$ must be all of $E$, as desired. 
\end{proof}

\section{Local correspondence of knots}\label{sec:localcorr}

The goal of this section is to prove Theorem \ref{thm:second}, the local correspondence: any isomorphism of absolute Galois groups sends decomposition groups to unique decomposition groups. Theorem \ref{thm:localcorrespondence} combines the injectivity of Corollary \ref{cor:geninjective} and the surjectivity of Corollary \ref{cor:surjgeneral} to show that any closed subgroup isomorphic to $\widehat{\mathbb{Z}}^2$ is contained in a unique decomposition group; uniqueness comes from Corollary \ref{cor:intersection}. Theorem \ref{thm:second} combines these results to obtain equality of decomposition groups and invokes the characteristic-preserving condition to describe the full local correspondence.

\medskip

Fix a stably Chebotarev link $\mathcal{L} \subset S^3$. 

\begin{prop}\label{prop:uniqueness}
Let $h: \widetilde{M} \to S^3$ be an intermediate branched cover of the universal $\mathcal{L}$-branched cover $\overline{M_{\mathcal{L}}}$. If two distinct knots are totally indecomposable with respect to the covering $\overline{M_{\mathcal{L}}} \to \widetilde{M}$, then $\widetilde{M}=\overline{M_{\mathcal{L}}}$ as covering maps. 
\end{prop}

\begin{proof}
Denote the covering $\overline{h}: \overline{M_{\mathcal{L}}} \to \widetilde{M}$. Suppose $K_1 \neq K_2 \in \mathcal{L}_{\widetilde M}$ are totally indecomposable with respect to $\overline{h}$, so we have $D_{\overline{K_1}|K_1} = \Gal(\overline{h}) = D_{\overline{K_2}|K_2}$. By Theorem \ref{thm:profinitetorus}, these groups are isomorphic to closed subgroups of $\widehat{\mathbb{Z}}^2$. Fix a prime $p$. By Corollary \ref{cor:surjgeneral}, the natural restriction map
\begin{equation}
    \psi^1: H^1(\Gal(\overline{h}), \mathbb{F}_p) \twoheadrightarrow H^1(D_{\overline{K_1}|K_1}, \mathbb{F}_p) \times H^1(D_{\overline{K_2}|K_2}, \mathbb{F}_p)
\end{equation}
is a surjection. This is only possible if $H^1(\Gal(\overline{h}), \mathbb{F}_p)$ is trivial, since the left side is isomorphic to both components of the second side. This is true for any prime $p$, so it follows that $\Gal(\overline{h})$ is trivial, and so $\overline{M_\mathcal{L}} = \widetilde{M}$.
\end{proof}

\begin{corollary}\label{cor:intersection}
Let $h: M \to S^3 \in \mathrm{Cov}(S^3, \mathcal{L})$. Given any $\overline{K_1} \neq \overline{K_2}$ over knots $K_1, K_2 \in \mathcal{L}_M$, the intersection $D_{\overline{K_1}|K_1} \cap D_{\overline{K_2}|K_2}$ is trivial.
\end{corollary}

\begin{proof}
Let $\overline{h_{12}}: \overline{M_{\mathcal{L}}} \to \widetilde{M_{12}}$ be the covering corresponding to $D_{\overline{K_1}|K_1} \cap D_{\overline{K_2}|K_2}$, and let $\widetilde M_i$ denote the subcover corresponding to $D_{\overline{K_i}|K_i}$ with coverings $\overline{h}_i: \overline{M_{\mathcal{L}}} \to \widetilde M_i$. Let $K_i' = \overline{h}_i(\overline{K}_i)$, by definition of $\widetilde{M}_i$, the knots $K'_i$ are totally indecomposable in $\overline{h}_i$. The knots $K_i'' = \overline{h_{12}}(\overline K_i)$ are distinct: if $\overline{K_2} = u\overline{K_1}$ with $u \in D_{\overline{K_1}|K_1} \cap D_{\overline{K_2}|K_2}$, then $u$ fixes $\overline{K_1}$, so $\overline{K_2} = \overline{K_1}$. By Proposition \ref{prop:uniqueness}, it follows that $\overline{h_{12}}$ is trivial, hence the intersection $D_{\overline{K_1}|K_1} \cap D_{\overline{K_2}|K_2}$ is trivial.
\end{proof}

\begin{corollary}\label{cor:passtofinitecovers}
Let $H$ be an infinite closed subgroup of $\Gal(M,\mathcal{L}_M)$. If for some knot $\overline{K}\in \overline{M_{\mathcal{L}}}$, $(H:H\cap D_{\overline{K}|K})<\infty$, then $H\subset D_{\overline{K}|K}$.
\end{corollary}

\begin{proof}
Let $\widetilde{M_H}$ be the cover corresponding to $H$ and $\widetilde{M_U}$ be the cover corresponding to some open subgroup of $H\cap D_{\overline{K}|K}$ which is normal in $H$, for instance the normal closure of $H \cap D_{\overline K|K}$ in $H$, which is open in $H$ since $[H:H \cap D_{\overline K|K}] < \infty$. Let $\bar{h}_H:\overline{M_{\mathcal{L}}}\rightarrow \widetilde{M_H}$ denote the universal covering of $\widetilde{M_H}$, and similarly denote by $\bar{h}_U$ the universal covering of $\widetilde{M_U}$. By the definition of $\widetilde{M_U}$, the knot $\bar{h}_U(\overline{K})$ is totally indecomposable in $\overline{M_{\mathcal{L}}}$. On the other hand, since $\widetilde{M_U} \to \widetilde{M_H}$ is normal, the knots above $\bar{h}_H(\overline{K})$ in $\widetilde{M_U}$ are the $\bar h_U(g\overline K)$ for $g \in H$, and $U = gUg^{-1} \subset D_{g\overline K|K}$, so each of them is also totally indecomposable in $\bar{h}_U$. Since $U$ is open in $H$ and $H$ is infinite, $U$ cannot be trivial, and so by Proposition \ref{prop:uniqueness}, we can only have a single knot above $\bar{h}_H(\overline{K})$ in $\widetilde{M_U}$. Since these knots correspond to the cosets $H/(H \cap D_{\overline K|K})$, this gives $H = H \cap D_{\overline K|K}$, i.e. $H \subset D_{\overline K|K}$.
\end{proof}

\begin{theorem}\label{thm:localcorrespondence}
Let $h: M \to S^3 \in \mathrm{Cov}(S^3, \mathcal{L})$. Given $\widehat{\mathbb{Z}}^2 \cong G \subset \Gal(M, \mathcal{L}_M)$, there exists a unique knot $K \in \mathcal{L}_M$ such that $G \subset D_{\overline{K}|K}$.
\end{theorem}

\begin{proof}
Given $\widehat{\mathbb{Z}}^2 \cong G \subset \Gal(M, \mathcal{L}_M)$, take the branched subcover $h_G: \widetilde{M_G} \to M$ of $\overline{M_\mathcal{L}}$ corresponding to $G$. Let $\mathcal{L}_G = h_G^{-1}(\mathcal{L}_M)$. We now show that there exists a unique knot in $\mathcal{L}_G$ which is totally indecomposable in $\overline{M_\mathcal{L}}$, i.e., $G \subset D_{\overline{h_G(K)}|h_G(K)}$; this suffices to prove the theorem. By Corollary \ref{cor:geninjective}, we have an injection, for any prime $p$,
\begin{equation}
     H^2(\Gal(\widetilde{M_G}, \mathcal{L}_G), \mathbb{F}_p) \hookrightarrow \prod_{K \in \mathcal{L}_G} H^2(D_{\overline{K}|K}, \mathbb{F}_p).
\end{equation}
Thus, there must be at least one knot $K$ for which $H^2(D_{\overline{K}|K}, \mathbb{F}_p)$ is nontrivial. By Corollary \ref{cor:passtofinitecovers} applied to $H = G$, it suffices to show that $G'' \subset D_{\overline K|K}$ for some open subgroup $G'' \subset G$, since then $[G:G \cap D_{\overline K|K}] < \infty$; we will in fact show this for $G'' = G$. Since $G$ is torsion-free, so is $D_{\overline K|K}$, whose pro-$p$ part is therefore $\mathbb Z_p^r$ with $r \leq 2$; as $H^2(D_{\overline K|K}; \mathbb F_p) \neq 0$ and $H^2(\mathbb Z_p; \mathbb F_p) = 0$, we get $r = 2$, hence $\dim_{\mathbb F_p}H^1(D_{\overline K|K}, \mathbb F_p) = 2$. We now claim that for exactly one knot $K \in \mathcal{L}_G$, $\dim_{\mathbb{F}_p}H^1(D_{\overline{K}|K}, \mathbb{F}_p) \geq 2$. Let $h': \widetilde{M'} \to \widetilde{M_G}$ be a finite branched cover corresponding to a finite-index subgroup of $G$. By Corollary \ref{cor:surjgeneral}, we have a surjection
\begin{equation}
    H^1(\Gal(\widetilde{M'}, (h')^{-1}(\mathcal{L}_G)), \mathbb{F}_p) \twoheadrightarrow \prod_{K' \in (h')^{-1}(K)}H^1(D_{\overline{K'}|K'}, \mathbb{F}_p).
\end{equation}
However, $\Gal(\widetilde{M}', (h')^{-1}(\mathcal{L}_G))$ is a finite-index subgroup of $G \cong \widehat{\mathbb{Z}}^2$, and hence the dimension of the left side over $\mathbb{F}_p$ is equal to 2. For any lift $K' \in (h')^{-1}(K)$, the group $D_{\overline K'|K'}$ is a finite-index subgroup of $D_{\overline K|K} \cong \widehat{\mathbb Z}^2$, so they share the same dimension over $\mathbb F_p$. Therefore, for every such lift, $\dim_{\mathbb{F}_p}H^1(D_{\overline{K'}|K'}, \mathbb{F}_p) = \dim_{\mathbb{F}_p}H^1(D_{\overline{K}|K}, \mathbb{F}_p) \geq 2$. By counting dimensions over $\mathbb{F}_p$, we conclude that there can only be one $K'$ lying over $K$. Therefore, $K$ is totally indecomposable in any finite branched cover of $\widetilde{M_G}$, and hence also in $\overline{M_\mathcal{L}}$.
So there exists a totally indecomposable knot in $\mathcal{L}_G$. Combined with Proposition \ref{prop:uniqueness}, which guarantees at most one such knot, there is exactly one totally indecomposable knot in $\mathcal{L}_G$. If there exist two distinct knots $K_1, K_2 \in \mathcal{L}_M$ such that $G \subset D_{\overline{K_i}|K_i}$ for $i = 1, 2$, this would imply a nontrivial intersection between two distinct decomposition groups, a contradiction of Corollary \ref{cor:intersection}.
\end{proof}

We now upgrade from containment to equality and establish the full bijection on decomposition groups.

\begin{theorem}\label{thm:second}[Theorem \ref{thm:main}(1)]
Let $h_i: M_i \to S^3, i = 1, 2$ be two coverings of $S^3$ branched over finite sublinks $L_i \subset \mathcal{L}$. Suppose $\sigma: \Gal(M_1, \mathcal{L}_{M_1}) \to \Gal(M_2, \mathcal{L}_{M_2})$ is an isomorphism. Given a knot $K_1 \in \mathcal{L}_{M_1}$: 
\begin{enumerate}
    \item $\sigma(D_{\overline{K_1}|K_1}) = D_{\overline{K_2}|K_2}$ for a unique $K_2 \in \mathcal{L}_{M_2}$.
    \item If $\sigma$ is characteristic-preserving, then $K_2 \in h_2^{-1}(h_1(K_1))$.
\end{enumerate}
\end{theorem}

\begin{proof}
\begin{enumerate}
    \item Let $U_i = \Gal(M_i, \mathcal{L}_{M_i})$, viewed as open subgroups of $G = \Gal(S^3, \mathcal{L})$. Given $K_1 \in \mathcal{L}_{M_1}$, the image $\sigma(D_{\overline{K_1}|K_1}) \cong \widehat{\mathbb{Z}}^2$ is a subgroup of $U_2$. By Theorem \ref{thm:localcorrespondence}, there exists a unique $K_2 \in \mathcal{L}_{M_2}$ such that $\sigma(D_{\overline{K_1}|K_1}) \subset D_{\overline{K_2}|K_2}$.
    
    \medskip
    
    We claim that $\sigma(D_{\overline{K_1}|K_1}) = D_{\overline{K_2}|K_2}$. Since $D_{\overline{K_2}|K_2} \cong \widehat{\mathbb{Z}}^2$, applying Theorem \ref{thm:localcorrespondence} to $\sigma^{-1}$ gives a unique $K_1' \in \mathcal{L}_{M_1}$ with $\sigma^{-1}(D_{\overline{K_2}|K_2}) \subset D_{\overline{K_1'}}$. Now $D_{\overline{K_1}|K_1} \subset \sigma^{-1}(D_{\overline{K_2}|K_2}) \subset D_{\overline{K_1'}}$, so $D_{\overline{K_1}|K_1} \cap D_{\overline{K_1'}} \supset D_{\overline{K_1}|K_1} \neq 1$. By Corollary \ref{cor:intersection}, $\overline{K_1'} = \overline{K_1}$. Thus $\sigma^{-1}(D_{\overline{K_2}|K_2}) \subset D_{\overline{K_1}|K_1}$, which gives $D_{\overline{K_2}|K_2} \subset \sigma(D_{\overline{K_1}|K_1})$. Combined with $\sigma(D_{\overline{K_1}|K_1}) \subset D_{\overline{K_2}|K_2}$, we obtain $\sigma(D_{\overline{K_1}|K_1}) = D_{\overline{K_2}|K_2}$. Replacing $\overline K_1$ by $u\overline{K_1}$, $u \in U_1$, replaces $\overline{K_2}$ by $\sigma(u)\overline{K_2}$, so $K_2$ depends only on $K_1$. The same construction for $\sigma^{-1}$ is inverse to $K_1 \mapsto K_2$, so $\sigma_{\mathcal L}$ is a bijection.
    \item By (1) and the definition of characteristic-preserving (Definition \ref{def:charpres}), $K_2 = \sigma_{\mathcal L}(K_1)$ lies over $h_1(K_1)$.
\end{enumerate} 
\end{proof}

\section{Proof of the main theorem}\label{sec:proof}

We now synthesise the local-global principles of Section \ref{sec:Local-Global} with the Chebotarev density results of \ref{sec:totsplit} to prove Theorem \ref{thm:main}. The first part of the main theorem follows from Theorem \ref{thm:second}. The proof proceeds in three claims. First, we show that $\sigma$ induces conjugation by a single element at every finite level (Claim 1), using the embedding problem of Lemma \ref{lma:embdding} and the argument from Uchida \cite{uchida}. Second, we use compactness to produce a global element $\alpha \in \Gal(S^3, \mathcal{L})$ inducing $\sigma$ (Claim 2). We finally prove uniqueness (Claim 3).  

\medskip

Fix a stably Chebotarev link $\mathcal{L} \subset S^3$ with Galois group $G = \Gal(S^3, \mathcal{L})$. Let $h_i: M_i \to S^3 \in \Cov(S^3, \mathcal{L})$, $i = 1, 2$ be two coverings branched over finite sublinks $L_i \subset \mathcal{L}$. Let $U_i = \Gal(M_i, \mathcal{L}_{M_i})$ and $\sigma: U_1 \to U_2$ be a characteristic-preserving isomorphism. For any finite normal cover $h: N \to S^3 \in \mathrm{Cov}(S^3, \mathcal{L})$ through which both $M_1$ and $M_2$ factor, write $V = \Gal(N, \mathcal{L}_N) \triangleleft G$ so that $\Gal(h) \cong G/V$, and let $h_i': N \to M_i$ denote the intermediate covering maps.

\begin{theorem}\label{thm:sigmaV}
For any open normal subgroup $V \subset G$ contained in $U_1 \cap U_2$, we have $\sigma(V) = V$.
\end{theorem}

\begin{proof}
Let $N\rightarrow M_1$ be the cover of $M_1$ corresponding to $V$, and $N^\sigma\rightarrow M_2$ the covering corresponding to $\sigma(V)$. Denote by $N'$ the normal closure over $S^3$ of $N^\sigma$. Let $K\in\mathcal{L}$ be a knot which is totally split in $N$, and put $n=\deg(N\rightarrow S^3)$. The decomposition groups of the $[G:U_1]$ knots of $\mathcal L_{M_1}$ over $K$ lie in $V$. By Theorem \ref{thm:second} and characteristic preservation, $\sigma$ carries them to the decomposition groups of $[G:U_1]$ distinct knots of $\mathcal L_{M_2}$ over $K$, which lie in $\sigma(V)$, which is normal in $U_2$, so these knots are totally split in $N^\sigma \to M_2$. Hence $K$ has at least $[G:U_1][U_1:V] = n$ knots above it in $N^\sigma$, so $n \leq n' := \deg(N^\sigma \to S^3)$. By Lemma \ref{lma:densitytotallysplit} applied to $N'$, there exists a knot in $\mathcal{L}$ which is totally split in $N'$. The same argument for $\sigma^{-1}$ shows that knot has at least $n'$ knots above it in $N$, so $n' \leq n$. Thus, $n' = n$. Hence $[G:U_1] = [G:U_2]$, and the count above now shows that for any open $W \subset U_1$ normal in $G$, a knot totally split in the cover of $W$ has $[G:\sigma(W)]$ preimages in the cover of $\sigma(W)$, and symmetrically for $\sigma^{-1}$ and any open $W \subset U_2$ normal in $G$. In addition, we have shown that any knot which is totally split in $N$ is also totally split in $N^\sigma$.

\medskip

Now suppose $K$ is totally split in $N^\sigma$. It is therefore totally split in its Galois closure $N'$, with corresponding subgroup $V' \subset G$. Let $(N')^{\sigma^{-1}}$ be the cover corresponding to $\sigma^{-1}(V')$. By the above argument applied to $W = V' \subset U_2$ and $\sigma^{-1}$, any knot which is totally split in $N'$ will be totally split in the normal closure $N''$ of $(N')^{\sigma^{-1}}$.

\medskip

Putting the above two arguments together, any knot $K$ which is totally split in $N$ is also totally split in $N''$. By Theorem \ref{thm:subcover}, $N''$ is a subcover of $N$. By construction, $N''$ is a cover of $N$, so we have $N = N''$, and so $N' = N^\sigma$, so $N^\sigma$ is normal. 

\medskip

We have also shown that the sets of totally split knots in $N$ and $N^\sigma$ completely coincide. By Corollary \ref{cor:first}, $N = N^\sigma$, so $V = \sigma(V)$.
\end{proof}

This theorem indicates that for $N$ a finite cover normal over $S^3$ and covering $M_1$, $M_2$, $\sigma$ descends to an isomorphism $\sigma_N: \Gal(h_1') \to \Gal(h_2')$.

\begin{lemma}\label{lma:uchida2}
Let $h: N \to S^3$ be a finite normal cover factoring through $M_1$ and $M_2$, and let $g \in \Gal(h_1')$. Then $\langle g \rangle$ and $\langle \sigma_N(g) \rangle$ are conjugate in $\Gal(h)$. In particular, there exists some integer $r$ relatively prime to the order of $g$ such that $\sigma_N(g)$ is conjugate to $g^r$ in $\Gal(h)$. 
\end{lemma}

\begin{proof}
Let $L_0 \subset \mathcal{L}$ be the branch locus of $h: N \to S^3$. Since $g \in U_1/V \subset G/V$, by Lemma \ref{lma:chebo} there exists $K \in \mathcal L \setminus L_0$ with Frobenius conjugate to $g$ in $G/V$. Choose a lift $K'$ of $K$ to $N$ with $\text{Frob}_{K'|K} = g$. Since $K \notin L_0$, $K$ is unramified, so $D_{K'|K} = \langle g \rangle$. Since $\langle g \rangle \subset U_1/V$, the image $K_1$ of $K'$ in $M_1$ has degree 1 over $K$. By Theorem \ref{thm:second} and the characteristic-preserving condition, $\sigma(D_{\overline{K_1}|K_1}) = D_{\overline{K_2}|K_2}$ for a unique $K_2 \in h_2^{-1}(K)$. By Theorem \ref{thm:sigmaV}, $\sigma(V) = V$, so $\sigma_N$ maps $D_{K'|K_1} = \langle g \rangle$ onto $D_{K''|K_2} = \langle \sigma_N(g) \rangle$ for some lift $K''$ of $K_2$ to $N$. Since $|\langle \sigma_N(g) \rangle| = |\langle g \rangle| = |D_{K'|K}|$, and $D_{K''|K_2} \subset D_{K''|K}$ with $|D_{K''|K}| = |D_{K'|K}|$ (conjugate decomposition groups at lifts of the same base knot), the inclusion $D_{K''|K_2} \subset D_{K''|K}$ is an equality. Then $D_{K''|K}$ is conjugate to $D_{K'|K} = \langle g \rangle$ in $G/V$, so $\langle g \rangle$ and $\langle \sigma_N(g) \rangle$ are conjugate in $\Gal(h)$.
\end{proof}

For the convenience of the reader, we state here again Theorem \ref{thm:main}. 

\begin{theorem}[Theorem \ref{thm:main}]
Let $\mathcal{L} \subset S^3$ be a stably Chebotarev link. Let $h_i: M_i \to S^3 \in \Cov(S^3, \mathcal{L})$, $i = 1, 2$, be two coverings of $S^3$ branched over a finite sublink $L_i \subset \mathcal{L}$. Then:
\begin{enumerate}
    \item Any isomorphism $\sigma: \Gal(M_1, \mathcal{L}_{M_1}) \to \Gal(M_2, \mathcal{L}_{M_2})$ induces a bijection $\sigma_{\mathcal L}: \mathcal{L}_{M_1} \rightarrow \mathcal{L}_{M_2}$.
    \item For any isomorphism $\sigma$ of absolute Galois groups which preserves characteristic, there exists a unique $\alpha\in\Gal(S^3,\mathcal{L})$ with $\sigma(u) = \alpha u \alpha^{-1}$ for all $u \in \Gal(M_1, \mathcal L_{M_1})$; in particular, $\alpha$ induces an isomorphism of covers $M_1 \to M_2$. By Remark \ref{rem:inducedIsom}, two branched covers over $S^3$ are isomorphic if and only if there exists a characteristic-preserving isomorphism between their absolute Galois groups. 
\end{enumerate} 
\end{theorem}

The first part was proven in Theorem \ref{thm:second}. We prove the second part by synthesizing our existing results to invoke Uchida's argument in \cite{uchida}, which is purely group-theoretic.

\begin{proof}[Proof of Theorem \ref{thm:main}(2)]
For any $h: N \to S^3$ a normal finite cover of $M_1$ and $M_2$, we wish to construct a global element $\alpha^N \in G = \Gal(S^3, \mathcal{L})$ that induces the finite-level isomorphism $\sigma_N$. Let $n = |\Gal(h)|$, and pick a prime number $p \equiv 1$ (mod $n$). By Lemma \ref{lma:embdding}, the split embedding problem for $\mathbb{F}_p[\Gal(h)]$ admits a proper topological solution, i.e., there exists a surjection $\psi: G \to E$ such that the diagram from that lemma commutes. This guarantees the existence of a normal branched covering $\widetilde{h}: \widetilde{N} \to S^3 \in \Cov(S^3, \mathcal{L})$ factoring through $N$, such that the intermediate cover $\widetilde{h'}: \widetilde{N} \to N$ has Galois group isomorphic to the group ring $\mathbb{F}_p[\Gal(h)]$, and $\Gal(\widetilde{h}) = E$. Let $V_{\widetilde{N}} = \Gal(\widetilde{N}, \mathcal{L}_{\widetilde{N}})$; this is the kernel of $\psi$, hence normal. Since $V_{\widetilde{N}} \subset U_1 \cap U_2$, Theorem \ref{thm:sigmaV} gives $\sigma(V_{\widetilde{N}}) = V_{\widetilde{N}}$. Consequently, $\sigma$ induces an isomorphism at this finite level: $\sigma_{\widetilde{N}}: U_1/V_{\tilde N} \to U_2/V_{\tilde N}$. Since $\sigma(V_{\widetilde{N}}) = V_{\widetilde{N}}$, the isomorphism $\sigma$ descends to an isomorphism $\sigma_N: \Gal(h_1') \to \Gal(h_2')$, and restricts to an automorphism of the abelian kernel $A :=\mathbb{F}_p[\Gal(h)]$. 
\begin{equation}
\begin{tikzcd}[row sep=2.5em, column sep=2em]
& \widetilde{N} \arrow[d, "{\mathbb{F}_p[G/V]}"'] & & & V_{\widetilde{N}} \arrow[d, hook] \\
& N \arrow[dl, "{U_1/V}"'] \arrow[dr, "{U_2/V}"] & & & V \arrow[dl, hook] \arrow[dr, hook] \\
M_1 \arrow[dr, "h_1"'] & & M_2 \arrow[dl, "h_2"] & U_1 \arrow[dr, hook] & & U_2 \arrow[dl, hook] \\
& S^3 & & & G &
\end{tikzcd}
\end{equation}
\underline{Claim 1: There exists some $g_0 \in \Gal(h)$ such that $\sigma_N(g) = g_0gg_0^{-1}$.}

\medskip

Every element of $\mathbb{F}_p[\Gal(h)]$ can be written as $\sum a_gg$ with $a_g \in \mathbb{F}_p, g \in \Gal(h)$. If $g \in U_1/V$, we have 
\begin{equation}
    \sigma_{\widetilde{N}}(g) = \sigma_{\widetilde{N}}(g \cdot 1) = \sigma_N(g)\sigma_{\widetilde{N}}(1)
\end{equation}
because the operation of $g$ on $\mathbb{F}_p[\Gal(h)]$ is induced from an inner automorphism by an element of $U_1/V_{\widetilde{N}}$. Consider $\mathbb{F}_p[U_1/V]$, as a subring of $\mathbb{F}_p[\Gal(h)]$. For any element $\alpha = \sum a_gg \in \mathbb{F}_p[U_1/V]$, we have
\begin{equation}
    \sigma_N(\alpha) = \sum a_g\sigma_N(g).
\end{equation}
Thus,
\begin{equation}
    \sigma_{\widetilde{N}}(\alpha) = \sigma_N(\alpha)\sigma_{\widetilde{N}}(1).
\end{equation}
for $\alpha \in \mathbb{F}_p[U_1/V]$. Let $\epsilon$ be an idempotent of $\mathbb{F}_p[U_1/V]$. Since $A\epsilon$ corresponds to an open normal subgroup of $G$ containing $V_{\widetilde N}$ and contained in $V \subset U_1 \cap U_2$, it follows from Theorem \ref{thm:sigmaV} that $A\epsilon$ is $\sigma_{\widetilde{N}}$-invariant, i.e., $\sigma_{\widetilde{N}}(A\epsilon) = A\epsilon$. Hence, 
\begin{equation}
    \sigma_{\widetilde{N}}(\epsilon) = \sigma_N(\epsilon)\sigma_{\widetilde{N}}(1) = \beta\epsilon
\end{equation}
for some $\beta \in \mathbb{F}_p[\Gal(h)]$. Since $1 - \epsilon$ is also idempotent in $\mathbb{F}_p[U_1/V]$, 
\begin{equation}
    \sigma_{\widetilde{N}}(1 - \epsilon) = \gamma(1 - \epsilon)
\end{equation}
for some $\gamma \in \mathbb{F}_p[\Gal(h)]$. Thus,
\begin{equation}   
    \sigma_{\widetilde{N}}(1) =  \sigma_{\widetilde{N}}(\epsilon) + \sigma_{\widetilde{N}}(1 - \epsilon) = \beta\epsilon + \gamma(1 - \epsilon)
\end{equation}
Multiplying $\epsilon$ on the right, we have 
\begin{equation}
    \sigma_{\widetilde{N}}(\epsilon) = \sigma_{N}(\epsilon)\sigma_{\widetilde{N}}(1) = \beta\epsilon = \sigma_{\widetilde{N}}(1)\epsilon.
\end{equation}
Lemma \ref{lma:uchida2} applies to $g = 1 \in \mathbb{F}_p[\Gal(h)] \hookrightarrow E$, with $\widetilde{N}$ in the role of common normal cover. The element $1 \in \mathbb{F}_p[G] = V/V_{\widetilde{N}}$ lies in $U_1/V_{\widetilde{N}} = \Gal(\widetilde{N} \to M_1)$, so the lemma gives that $\langle 1 \rangle$ and $\langle \sigma_{\widetilde{N}}(1) \rangle$ are conjugate in $E = \Gal(\widetilde{h})$. The element 1 has order $p$ in the additive group of $\mathbb{F}_p[\Gal(h)]$. Hence $\sigma_{\widetilde{N}}(1)$ is conjugate to $r \cdot 1$ for some $r \in \mathbb{F}_p^\times$. It remains to identify the conjugates of 1 in $E$. In the semidirect product $E = \mathbb{F}_p[\Gal(h)] \rtimes \Gal(h)$, conjugation by $g \in \Gal(h)$ sends $1 \in \mathbb{F}_p[\Gal(h)]$ to $g \cdot 1 \cdot g^{-1} = g$, since $g$ acts on the group ring by left multiplication on the basis. Since $\mathbb{F}_p[\Gal(h)]$ is abelian, conjugation within it fixes 1. Conjugation by $g \in \Gal(h)$ gives $g$ by the left module structure, so the conjugacy class of 1 in $E$ is exactly $\Gal(h)$ viewed as basis elements of the group ring. Since $\sigma_{\widetilde{N}}(1)$ is conjugate to $r \cdot 1$ and the conjugates of 1 are $\Gal(h)$, we conclude that $\sigma_{\widetilde{N}}(1) = rg_0$ for some $g_0 \in \Gal(h)$. Thus
\begin{equation}
    \sigma_N(\epsilon)g_0 = g_0\epsilon
\end{equation}
for every idempotent $\epsilon \in \mathbb{F}_p[U_1/V]$. Let $g \in U_1/V$ and let $m$ be its order. Since $p \equiv 1$ (mod $m$), $\mathbb{F}_p$ contains a primitive $m$-th root of unity $\mu$. Then
\begin{equation}
    \epsilon_i = m^{-1}(1 + \mu^ig + \dots + \mu^{(m-1)i}g^{m-1})
\end{equation}
are idempotents in $\mathbb{F}_p[U_1/V]$. Thus, we have 
\begin{equation}
    m^{-1}(1 + \mu^i\sigma_N(g) + \dots + \mu^{(m-1)i}\sigma_N(g)^{m-1})g_0 = m^{-1}g_0(1 + \mu^ig + \dots + \mu^{(m-1)i}g^{m-1}).
\end{equation}
Multiplying $\mu^{-i}$ and summing over all $i$, we get
\begin{equation}
    \sigma_N(g)g_0 = g_0g
\end{equation}
and hence
\begin{equation}
    \sigma_N(g) = g_0gg_0^{-1}
\end{equation}
as desired.

\medskip

\underline{Claim 2: There exists $\alpha \in G$ whose restriction to any finite Galois cover $N$ induces $\sigma_N$.}

\medskip

For each finite normal cover $h_N: N \to S^3 \in \Cov(S^3, \mathcal{L})$ covering $M_1$ and $M_2$, with $V_N = \Gal(N, \mathcal L_N)$, let 
\begin{equation}
    S_N = \{g \in G \mid \sigma(u) \equiv gug^{-1} \ \mathrm{mod} \ {V_N} \text{ for all } u \in U_1\}.
\end{equation}
By Claim 1, $S_N$ is nonempty for every such $N$. Since $S_N$ is a union of cosets of $V_N$ (the kernel of $G \to \Gal(h_N)$) in $G$, it is a nonempty compact subset of $G$. Given finitely many such $N_1, \dots, N_m$, let $N$ be their compositum, a finite normal cover dominating each $N_i$. Then $S_{N_1} \cap \dots \cap S_{N_m} \supset S_N$, which is nonempty by Claim 1. Therefore, $\{S_N\}$ has the finite intersection property. By compactness, their intersection $\bigcap_NS_N$ is non-empty. Any element $\alpha \in \bigcap_N S_N$ globally induces the isomorphism $\sigma$.

\medskip

\underline{Claim 3: $\alpha$ is unique.} 

\medskip

Suppose $\alpha$ and $\beta$ both induce $\sigma$. Then $c = \alpha^{-1}\beta$ centralizes $U_1$. For every compatible system $\bar K$, the subgroup $D_{\bar K\mid K} \cap U_1$ is open in $D_{\bar K\mid K} \cong \widehat{\mathbb Z}^2$ and hence contains some $d \neq 1$. Since $c$ centralizes $d$, $d = cdc^{-1}$. By Corollary \ref{cor:intersection}, $c\bar K = \bar K$. Applying Corollary \ref{cor:intersection} to two distinct knots in the universal $\mathcal L$-branched cover gives $c = 1$, so $\alpha = \beta$.

\end{proof}

\section{Further questions on links behaving like primes}\label{sec:future}

\subsection{Characteristic-preserving isomorphisms}\label{sec:charpres}

We highlight several differences between links and primes within the arithmetic topology dictionary which remain unaddressed. Further study of these phenomena would deepen the connection between links and primes. 

\begin{enumerate}
    \item Prime numbers are intrinsically different to each other in size; group-theoretically, one can distinguish them via the maximal Galois group ramified outside $p$. Although this is similar to how prime knots are determined by their fundamental groups, in stably Chebotarev links, we do not require that the knot components are prime knots or distinct; in the planetary link of the figure-eight knot, all topological link types exist \cite{ghrist}.
    \item  When distinguishing decomposition groups in number fields, an important tool is decomposing decomposition groups into pro-$p$ components. This is clear in the \emph{wild inertia group}, but recently \cite{karshon2026pro} it has worked when considering tame inertia. One way to see this, comes from the fact that an abelian profinite group has an intrinsic $\widehat{\mathbb{Z}}$-module structure, which is equipped with an action of all distinct primes $p$. So $\mathbb{Z}$ and primes play two different roles. On the other hand, absolute decomposition groups of knots do not possess this extra module structure; the extra role of $\mathbb{Z}$ is not seen. 
    \item Certain knots like the figure-eight knot can be \emph{universal}, meaning that any closed orientable 3-manifold can be realized as a branched cover over that knot \cite{hilden}. In addition, every closed orientable 3-manifold can be a 3-fold branched covering space over a link in $S^3$ \cite{hilden3fold}. The analogous statements in number theory are false, indicating that branched coverings possess less intrinsic rigidity than the number field extensions. 
\end{enumerate} 

The above three observations are all reflections of the fact that within the arithmetic topology framework used in this paper, the topology of knots alone lacks the rigidity of primes in number fields. These phenomena indicate a need for some additional structure beyond pure topology to force rigidity. The characteristic-preserving condition of Theorem \ref{thm:main} provides what is needed to supply the missing rigidity at the level of profinite group isomorphisms. 

\medskip

The extra structure required for full rigidity may also be provided by geometry, particularly hyperbolic geometry, as speculated by Mazur \cite{mazur}. Xu \cite{xu} showed that peripheral subgroups are detected by the profinite completion, and in a recent preprint \cite{Xu26}, showed that fundamental groups of finite-volume hyperbolic 3-manifolds are relatively profinitely rigid. One hopes that decomposition groups may be distinguished when Chebotarev links which are ``sufficiently hyperbolic".

\medskip

A correct analogue of the full set of prime numbers should be a stably Chebotarev link $\mathcal{L}$ that intrinsically satisfies such rigidity, and in particular should have the property that isomorphisms between finite-index open subgroups are characteristic-preserving. For such a link, our Neukirch--Uchida theorem will apply for any group theoretic isomorphism, reflecting the full number-theoretic situation. 

\medskip

It would be interesting to know the extent to which stably Chebotarev links behave or do not behave like the set of primes in this sense. We pose the following question.

\begin{question}
\begin{enumerate}
    \item Does there exist a stably Chebotarev link $\mathcal{L} \subset S^3$ such that all isomorphisms between two open subgroups of $\Gal(S^3, \mathcal{L})$ preserve characteristic?
    \item Does there exist a stably Chebotarev link $\mathcal{L} \subset S^3$ and an isomorphism between two open subgroups of $\Gal(S^3, \mathcal{L})$ which does not preserve characteristic? Does this isomorphism induce any topological relationship between the corresponding branched covers? 
\end{enumerate}
\end{question}

\subsection{Local-global principles and Poitou--Tate duality}

When one works with a finite link $L$ whose exterior $X$ is aspherical with $\pi_1$-injective boundary tori, the absolute Galois group $\Gal(M, L)$ is the profinite completion $\widehat{\pi}_1(X)$. Using the argument of Lemma \ref{lma:lefschetz} and Lefschetz duality, we get the full exact sequence:
\begin{align}
\begin{split}
    0 &\to H^0(\widehat{\pi}_1(X), A) \to H^0\left(B, A\right) \to H^2(\widehat{\pi}_1(X), A)^\vee \\ &\to H^1(\widehat{\pi}_1(X), A) \to H^1\left(B, A\right) \to H^1(\widehat{\pi}_1(X), A)^\vee \\\ &\to H^2(\widehat{\pi}_1(X), A) \to H^2\left(B, A\right) \to H^0(\widehat{\pi}_1(X), A)^\vee \to 0
\end{split}
\end{align}
where $B = \prod_{K \in \mathcal L}D_{\overline K|K}$. An analogue of the celebrated Poitou--Tate duality would be a duality sequence for a suitable stably Chebotarev link. Theorem \ref{thm:exact} establishes that for $A = \mathbb F_p$, we have an exact sequence
\begin{equation}
    H^1(\Gal(M, \mathcal{L}_M), A) \to P^1(\Gal(M, \mathcal{L}_M), A) \to H^1(\Gal(M, \mathcal{L}_M), A)^\vee
\end{equation}
where $P^1$ denotes the appropriate restricted product of local cohomology groups. This exact sequence can be viewed as the middle three terms of a 3-manifold Poitou--Tate sequence, and we pose the question of whether this extends to the full Poitou--Tate sequence. In addition, we hope to extend this result to more general coefficients, namely finite abelian $\Gal(M, \mathcal{L}_M)$-modules $A$. We thus pose a general question:

\begin{question}\label{q:poitoutate}
Let $G = \Gal(M, \mathcal{L}_M)$. Does the exact sequence of Theorem \ref{thm:exact} extend to a full topological Poitou--Tate sequence
\begin{align}
\begin{split}
    0 &\to H^0(G, A) \to P^0(G, A) \to H^2(G, A')^\vee \\ &\to H^1(G, A) \to P^1(G, A) \to H^1(G, A')^\vee \\ &\to H^2(G, A) \to P^2(G, A) \to H^0(G, A')^\vee \to 0?
\end{split}
\end{align}
\end{question}

\begin{remark}
The exact definition of $A'$ is unclear in the general setup. In number theory $A'=\Hom(A, \mathcal O_{K, S}^\times)$ with $n = |A|$, and the proof of duality uses the id\`ele class group, and the fact that it is a formation module (see \cite{poitou} and \cite{haberland1978galois}).
Although in \cite{uekiniibo} id\`ele class groups were discussed, they are not class formations in general (Remark 8.2 in \cite{uekiniibo}). Thus, we expect that either further conditions on the link $\mathcal{L}$ or further modifications of the notion of the id\`ele class group for 3-manifolds are needed.
\end{remark}

\bibliographystyle{abbrv}
\bibliography{refs}	

@book{neukirch2013cohomology,
  title={Cohomology of number fields},
  author={Neukirch, J{\"u}rgen and Schmidt, Alexander and Wingberg, Kay},
  volume={323},
  year={2013},
  publisher={Springer Science \& Business Media}
}

@book{morishita2012knots,
  title={Knots and primes},
  author={Morishita, Masanori},
  year={2012},
  publisher={Springer}
}

@article{ueki2021chebotarev,
  title={Chebotarev links are stably generic},
  author={Ueki, Jun},
  journal={Bulletin of the London Mathematical Society},
  volume={53},
  number={1},
  pages={82--91},
  year={2021},
  publisher={Wiley Online Library}
}

@book {haberland1978galois,
    AUTHOR = {Haberland, Klaus},
     TITLE = {Galois cohomology of algebraic number fields},
      NOTE = {With two appendices by Helmut Koch and Thomas Zink},
 PUBLISHER = {VEB Deutscher Verlag der Wissenschaften, Berlin},
      YEAR = {1978},
     PAGES = {145},
   MRCLASS = {12A60},
  MRNUMBER = {519872},
MRREVIEWER = {Chih-Han\ Sah},
}

@article{mcmullen2013knots,
  title={Knots which behave like the prime numbers},
  author={McMullen, Curtis T},
  journal={Compositio Mathematica},
  volume={149},
  number={8},
  pages={1235--1244},
  year={2013},
  publisher={London Mathematical Society}
}

@article {neukirch1968algebraische,
    AUTHOR = {Neukirch, J\"urgen},
     TITLE = {\"{U}ber eine algebraische {K}ennzeichnung der {H}enselk\"orper},
   JOURNAL = {J. Reine Angew. Math.},
  FJOURNAL = {Journal f\"ur die Reine und Angewandte Mathematik. [Crelle's
              Journal]},
    VOLUME = {231},
      YEAR = {1968},
     PAGES = {75--81},
      ISSN = {0075-4102,1435-5345},
   MRCLASS = {10.67},
  MRNUMBER = {227139},
MRREVIEWER = {T.\ Kubota},
       DOI = {10.1515/crll.1968.231.75},
       URL = {https://doi.org/10.1515/crll.1968.231.75},
}

@article{parry1990zeta,
  title={Zeta functions and the periodic orbit structure of hyperbolic dynamics},
  author={Parry, William and Pollicott, Mark},
  journal={Ast{\'e}risque},
  volume={187},
  number={188},
  pages={1--268},
  year={1990}
}

@article {uekiniibo,
    AUTHOR = {Niibo, Hirofumi and Ueki, Jun},
     TITLE = {Id\`elic class field theory for 3-manifolds and very
              admissible links},
   JOURNAL = {Trans. Amer. Math. Soc.},
  FJOURNAL = {Transactions of the American Mathematical Society},
    VOLUME = {371},
      YEAR = {2019},
    NUMBER = {12},
     PAGES = {8467--8488},
      ISSN = {0002-9947,1088-6850},
   MRCLASS = {57M12 (11R37 57M99)},
  MRNUMBER = {3955553},
MRREVIEWER = {Igor\ V.\ Nikolaev},
       DOI = {10.1090/tran/7480},
       URL = {https://doi.org/10.1090/tran/7480},
}

@article{mostow,
  title={Quasi-conformal mappings in $ n $-space and the rigidity of hyperbolic space forms},
  author={Mostow, George D},
  journal={Publications Math{\'e}matiques de l'IH{\'E}S},
  volume={34},
  pages={53--104},
  year={1968}
}

@incollection {mazur,
    AUTHOR = {Mazur, Barry},
     TITLE = {Primes, knots and {P}o},
 BOOKTITLE = {Essays on topology---dedicated to {V}alentin {P}o\'enaru},
     PAGES = {13--31},
 PUBLISHER = {Springer, Cham},
      YEAR = {2025},
      ISBN = {978-3-031-81413-6; 978-3-031-81414-3},
   MRCLASS = {57K10 (01A70 11R32)},
  MRNUMBER = {4944626},
       DOI = {10.1007/978-3-031-81414-3\_2},
       URL = {https://doi.org/10.1007/978-3-031-81414-3_2},
}

@article{mazur2,
    AUTHOR = {Mazur, Barry},
    TITLE = {Remarks on the Alexander polynomial},
 Journal = {Unpublished note},
 YEAR = {1964},
    URL = {https://bpb-us-e1.wpmucdn.com/sites.harvard.edu/dist/a/189/files/2023/01/Remarks-on-the-Alexander-Polynomial.pdf}
}

@article{mazur1973notes,
  title={Notes on {\'e}tale cohomology of number fields},
  author={Mazur, Barry},
  Journal={Annales scientifiques de l'{\'E}cole Normale Sup{\'e}rieure},
  volume={6},
  number={4},
  pages={521--552},
  year={1973}
}

@article {uekibranched,
    AUTHOR = {Ueki, Jun},
     TITLE = {On the homology of branched coverings of 3-manifolds},
   JOURNAL = {Nagoya Math. J.},
  FJOURNAL = {Nagoya Mathematical Journal},
    VOLUME = {213},
      YEAR = {2014},
     PAGES = {21--39},
      ISSN = {0027-7630,2152-6842},
   MRCLASS = {57M12 (11R29 11R32 12G05)},
  MRNUMBER = {3290684},
MRREVIEWER = {Igor\ V.\ Nikolaev},
       DOI = {10.1215/00277630-2393795},
       URL = {https://doi.org/10.1215/00277630-2393795},
}

@article {mihara,
    AUTHOR = {Mihara, Tomoki},
     TITLE = {Cohomological approach to class field theory in arithmetic
              topology},
   JOURNAL = {Canad. J. Math.},
  FJOURNAL = {Canadian Journal of Mathematics. Journal Canadien de
              Math\'ematiques},
    VOLUME = {71},
      YEAR = {2019},
    NUMBER = {4},
     PAGES = {891--935},
      ISSN = {0008-414X,1496-4279},
   MRCLASS = {55N20 (11R37 11Z05 18F15 57P05)},
  MRNUMBER = {3984024},
MRREVIEWER = {Anatoly\ N.\ Kochubei},
       DOI = {10.4153/cjm-2018-020-0},
       URL = {https://doi.org/10.4153/cjm-2018-020-0},
}

@book {washington,
    AUTHOR = {Washington, Lawrence C.},
     TITLE = {Introduction to cyclotomic fields},
    SERIES = {Graduate Texts in Mathematics},
    VOLUME = {83},
 PUBLISHER = {Springer-Verlag, New York},
      YEAR = {1982},
     PAGES = {xi+389},
      ISBN = {0-387-90622-3},
   MRCLASS = {11-01 (11R18 11R23)},
  MRNUMBER = {718674},
MRREVIEWER = {T.\ Mets\"ankyl\"a},
       DOI = {10.1007/978-1-4684-0133-2},
       URL = {https://doi.org/10.1007/978-1-4684-0133-2},
}

@article {neukirchcrelle,
    AUTHOR = {Neukirch, J\"urgen},
     TITLE = {Kennzeichnung der endlich-algebraischen {Z}ahlk\"orper durch
              die {G}aloisgruppe der maximal aufl\"osbaren {E}rweiterungen},
   JOURNAL = {J. Reine Angew. Math.},
  FJOURNAL = {Journal f\"ur die Reine und Angewandte Mathematik. [Crelle's
              Journal]},
    VOLUME = {238},
      YEAR = {1969},
     PAGES = {135--147},
      ISSN = {0075-4102,1435-5345},
   MRCLASS = {12.45},
  MRNUMBER = {258804},
MRREVIEWER = {John\ Labute},
       DOI = {10.1515/crll.1969.238.135},
       URL = {https://doi.org/10.1515/crll.1969.238.135},
}

@article {neukirchinvenciones,
    AUTHOR = {Neukirch, J\"urgen},
     TITLE = {Kennzeichnung der {$p$}-adischen und der endlichen
              algebraischen {Z}ahlk\"orper},
   JOURNAL = {Invent. Math.},
  FJOURNAL = {Inventiones Mathematicae},
    VOLUME = {6},
      YEAR = {1969},
     PAGES = {296--314},
      ISSN = {0020-9910,1432-1297},
   MRCLASS = {12.50},
  MRNUMBER = {244211},
MRREVIEWER = {John\ Labute},
       DOI = {10.1007/BF01425420},
       URL = {https://doi.org/10.1007/BF01425420},
}

@article {uchida,
    AUTHOR = {Uchida, K\^oji},
     TITLE = {Isomorphisms of {G}alois groups},
   JOURNAL = {J. Math. Soc. Japan},
  FJOURNAL = {Journal of the Mathematical Society of Japan},
    VOLUME = {28},
      YEAR = {1976},
    NUMBER = {4},
     PAGES = {617--620},
      ISSN = {0025-5645,1881-1167},
   MRCLASS = {12A55},
  MRNUMBER = {432593},
MRREVIEWER = {Jack\ Sonn},
       DOI = {10.2969/jmsj/02840617},
       URL = {https://doi.org/10.2969/jmsj/02840617},
}

@incollection {grothendieck,
    AUTHOR = {Grothendieck, Alexandre},
     TITLE = {Esquisse d'un programme},
 BOOKTITLE = {Geometric {G}alois actions, 1},
    SERIES = {London Math. Soc. Lecture Note Ser.},
    VOLUME = {242},
     PAGES = {5--48},
      NOTE = {With an English translation on pp.\ 243--283},
 PUBLISHER = {Cambridge Univ. Press, Cambridge},
      YEAR = {1997},
      ISBN = {0-521-59642-4},
   MRCLASS = {14H10 (14F20 14H30)},
  MRNUMBER = {1483107},
MRREVIEWER = {Hiroaki\ Nakamura and Yasutaka\ Ihara},
}

@article{mnt,
  title={The {G}rothendieck conjecture on the fundamental groups of algebraic curves},
  author={Nakamura, Hiroaki and Tamagawa, Akio and Mochizuki, Shinichi},
  journal={Sugaku Expositions},
  volume={14},
  number={1},
  pages={31--54},
  year={2001},
  publisher={Providence, RI, USA: The Society, c1988-}
}

@article {xu,
    AUTHOR = {Xu, Xiaoyu},
     TITLE = {Profinite almost rigidity in 3-manifolds},
   JOURNAL = {Adv. Math.},
  FJOURNAL = {Advances in Mathematics},
    VOLUME = {480},
      YEAR = {2025},
     PAGES = {Paper No. 110505, 79},
      ISSN = {0001-8708,1090-2082},
   MRCLASS = {57K30 (20F34 20F65 57M05 57M10 57M50)},
  MRNUMBER = {4953004},
       DOI = {10.1016/j.aim.2025.110505},
       URL = {https://doi.org/10.1016/j.aim.2025.110505},
}

@inproceedings {reid,
    AUTHOR = {Reid, Alan W.},
     TITLE = {Profinite rigidity},
 BOOKTITLE = {Proceedings of the {I}nternational {C}ongress of
              {M}athematicians---{R}io de {J}aneiro 2018. {V}ol. {II}.
              {I}nvited lectures},
     PAGES = {1193--1216},
 PUBLISHER = {World Sci. Publ., Hackensack, NJ},
      YEAR = {2018},
      ISBN = {978-981-3272-91-0; 978-981-3272-87-3},
   MRCLASS = {20E18 (57M07)},
  MRNUMBER = {3966805},
MRREVIEWER = {Gareth\ Wilkes},
}

@book {3groups,
    AUTHOR = {Aschenbrenner, Matthias and Friedl, Stefan and Wilton, Henry},
     TITLE = {3-manifold groups},
    SERIES = {EMS Series of Lectures in Mathematics},
 PUBLISHER = {European Mathematical Society (EMS), Z\"urich},
      YEAR = {2015},
     PAGES = {xiv+215},
      ISBN = {978-3-03719-154-5},
   MRCLASS = {57N10 (20F65 57-02 57M27)},
  MRNUMBER = {3444187},
MRREVIEWER = {Thomas\ Koberda},
       DOI = {10.4171/154},
       URL = {https://doi.org/10.4171/154},
}

@article {ghrist,
    AUTHOR = {Ghrist, Robert and Kin, Eiko},
     TITLE = {Flowlines transverse to knot and link fibrations},
   JOURNAL = {Pacific J. Math.},
  FJOURNAL = {Pacific Journal of Mathematics},
    VOLUME = {217},
      YEAR = {2004},
    NUMBER = {1},
     PAGES = {61--86},
      ISSN = {0030-8730,1945-5844},
   MRCLASS = {57M25 (37C10 37E99)},
  MRNUMBER = {2105766},
MRREVIEWER = {Quach thi C\^am V\^an},
       DOI = {10.2140/pjm.2004.217.61},
       URL = {https://doi.org/10.2140/pjm.2004.217.61},
}

@article {tamagawa,
    AUTHOR = {Tamagawa, Akio},
     TITLE = {The {G}rothendieck conjecture for affine curves},
   JOURNAL = {Compositio Math.},
  FJOURNAL = {Compositio Mathematica},
    VOLUME = {109},
      YEAR = {1997},
    NUMBER = {2},
     PAGES = {135--194},
      ISSN = {0010-437X,1570-5846},
   MRCLASS = {14H30 (11G20 14E20)},
  MRNUMBER = {1478817},
MRREVIEWER = {Hiroaki\ Nakamura},
       DOI = {10.1023/A:1000114400142},
       URL = {https://doi.org/10.1023/A:1000114400142},
}

@article {hilden,
    AUTHOR = {Hilden, Hugh M. and Lozano, Mar\'ia Teresa and Montesinos,
              Jos\'e{} Mar\'ia},
     TITLE = {On knots that are universal},
   JOURNAL = {Topology},
  FJOURNAL = {Topology. An International Journal of Mathematics},
    VOLUME = {24},
      YEAR = {1985},
    NUMBER = {4},
     PAGES = {499--504},
      ISSN = {0040-9383},
   MRCLASS = {57M25 (57M12 57N10)},
  MRNUMBER = {816529},
MRREVIEWER = {Wolfgang\ H.\ Heil},
       DOI = {10.1016/0040-9383(85)90019-9},
       URL = {https://doi.org/10.1016/0040-9383(85)90019-9},
}

@book {widder,
    AUTHOR = {Widder, David Vernon},
     TITLE = {The {L}aplace {T}ransform},
    SERIES = {Princeton Mathematical Series},
    VOLUME = {vol. 6},
 PUBLISHER = {Princeton University Press, Princeton, NJ},
      YEAR = {1941},
     PAGES = {x+406},
   MRCLASS = {42.4X},
  MRNUMBER = {5923},
MRREVIEWER = {J.\ D.\ Tamarkin},
}

@book {sga,
     TITLE = {Rev\^etements \'etales et groupe fondamental ({SGA} 1)},
    SERIES = {Documents Math\'ematiques (Paris)},
    VOLUME = {3},
 PUBLISHER = {Soci\'et\'e{} Math\'ematique de France, Paris},
    author={Grothendieck, Alexander},
      YEAR = {2003},
     PAGES = {xviii+327},
      ISBN = {2-85629-141-4},
   MRCLASS = {14E20 (14-06 14F35)},
  MRNUMBER = {2017446},
}

@article {wilkes,
    AUTHOR = {Wilkes, Gareth},
     TITLE = {Relative cohomology theory for profinite groups},
   JOURNAL = {J. Pure Appl. Algebra},
  FJOURNAL = {Journal of Pure and Applied Algebra},
    VOLUME = {223},
      YEAR = {2019},
    NUMBER = {4},
     PAGES = {1617--1688},
      ISSN = {0022-4049,1873-1376},
   MRCLASS = {20J06 (20E18 22C05 57M27)},
  MRNUMBER = {3906519},
MRREVIEWER = {Nadia\ P.\ Mazza},
       DOI = {10.1016/j.jpaa.2018.07.001},
       URL = {https://doi.org/10.1016/j.jpaa.2018.07.001},
}

@book {serre,
    AUTHOR = {Serre, Jean-Pierre},
     TITLE = {Galois cohomology},
      NOTE = {Translated from the French by Patrick Ion and revised by the
              author},
 PUBLISHER = {Springer-Verlag, Berlin},
      YEAR = {1997},
     PAGES = {x+210},
      ISBN = {3-540-61990-9},
   MRCLASS = {12G05 (11R34)},
  MRNUMBER = {1466966},
       DOI = {10.1007/978-3-642-59141-9},
       URL = {https://doi.org/10.1007/978-3-642-59141-9},
}

@incollection {hmm,
    AUTHOR = {Hillman, Jonathan and Matei, Daniel and Morishita, Masanori},
     TITLE = {Pro-{$p$} link groups and {$p$}-homology groups},
 BOOKTITLE = {Primes and knots},
    SERIES = {Contemp. Math.},
    VOLUME = {416},
     PAGES = {121--136},
 PUBLISHER = {Amer. Math. Soc., Providence, RI},
      YEAR = {2006},
      ISBN = {978-0-8218-3456-5; 0-8218-3456-8},
   MRCLASS = {57M25 (11R33)},
  MRNUMBER = {2276139},
       DOI = {10.1090/conm/416/07890},
       URL = {https://doi.org/10.1090/conm/416/07890},
}

@incollection {pop2011lectures,
    AUTHOR = {Pop, Florian},
     TITLE = {Lectures on anabelian phenomena in geometry and arithmetic},
 BOOKTITLE = {Non-abelian fundamental groups and {I}wasawa theory},
    SERIES = {London Math. Soc. Lecture Note Ser.},
    VOLUME = {393},
     PAGES = {1--55},
 PUBLISHER = {Cambridge Univ. Press, Cambridge},
      YEAR = {2012},
      ISBN = {978-1-107-64885-2},
   MRCLASS = {11R23 (11R32 14G32)},
  MRNUMBER = {2905528},
MRREVIEWER = {Alan\ Koch},
}

@article{karshon2026pro,
  title={Pro-$l$-by-cyclotomic and tamely ramified variants of the {N}eukirch-{U}chida Theorem},
  author={Karshon, Ido and Shusterman, Mark},
  journal={arXiv preprint arXiv:2601.01251},
  year={2026}
}

@article{gropper2023surfaces,
  title={Surfaces and p-adic fields I: Dehn twists},
  author={Gropper, Nadav},
  journal={arXiv preprint arXiv:2303.04309},
  year={2023}
}

@article{mochizuki1997version,
  title={A version of the {G}rothendieck conjecture for p-adic local fields},
  author={Mochizuki, Shinichi},
  journal={International Journal of Mathematics},
  volume={8},
  number={4},
  pages={499--506},
  year={1997},
  publisher={Singapore: World Scientific, c1990-}
}

@article{uchida1977isomorphisms,
  title={Isomorphisms of {G}alois groups of algebraic function fields},
  author={Uchida, K{\^o}ji},
  journal={Annals of Mathematics},
  volume={106},
  number={3},
  pages={589--598},
  year={1977}
}

@article{mathew2016galois,
  title={The {G}alois group of a stable homotopy theory},
  author={Mathew, Akhil},
  journal={Advances in Mathematics},
  volume={291},
  pages={403--541},
  year={2016},
  publisher={Elsevier}
}

@article {uekiiwasawa,
    AUTHOR = {Ueki, Jun},
     TITLE = {On the {I}wasawa {$\mu$}-invariants of branched
              {$\mathbb{Z}_p$}-covers},
   JOURNAL = {Proc. Japan Acad. Ser. A Math. Sci.},
  FJOURNAL = {Japan Academy. Proceedings. Series A. Mathematical Sciences},
    VOLUME = {92},
      YEAR = {2016},
    NUMBER = {6},
     PAGES = {67--72},
      ISSN = {0386-2194},
   MRCLASS = {57M12 (11R23 57M25)},
  MRNUMBER = {3508576},
MRREVIEWER = {Gabriel\ D.\ Villa-Salvador},
       DOI = {10.3792/pjaa.92.67},
       URL = {https://doi.org/10.3792/pjaa.92.67},
}

@article {uekikida,
    AUTHOR = {Ueki, Jun},
     TITLE = {On the {I}wasawa invariants for links and {K}ida's formula},
   JOURNAL = {Internat. J. Math.},
  FJOURNAL = {International Journal of Mathematics},
    VOLUME = {28},
      YEAR = {2017},
    NUMBER = {6},
     PAGES = {1750035, 30},
      ISSN = {0129-167X,1793-6519},
   MRCLASS = {57M27 (11R23 11S15 57M12 57M25 57M60)},
  MRNUMBER = {3663789},
MRREVIEWER = {L.\ Neuwirth},
       DOI = {10.1142/S0129167X17500355},
       URL = {https://doi.org/10.1142/S0129167X17500355},
}

@article {hilden3fold,
    AUTHOR = {Hilden, Hugh M.},
     TITLE = {Every closed orientable {$3$}-manifold is a {$3$}-fold
              branched covering space of {$S\sp{3}$}},
   JOURNAL = {Bull. Amer. Math. Soc.},
  FJOURNAL = {Bulletin of the American Mathematical Society},
    VOLUME = {80},
      YEAR = {1974},
     PAGES = {1243--1244},
      ISSN = {0002-9904},
   MRCLASS = {55A10 (57A10)},
  MRNUMBER = {350719},
MRREVIEWER = {J.\ S.\ Birman},
       DOI = {10.1090/S0002-9904-1974-13699-2},
       URL = {https://doi.org/10.1090/S0002-9904-1974-13699-2},
}

@article {ivanov,
    AUTHOR = {Ivanov, Alexander B.},
     TITLE = {On a generalization of the {N}eukirch-{U}chida theorem},
   JOURNAL = {Mosc. Math. J.},
  FJOURNAL = {Moscow Mathematical Journal},
    VOLUME = {17},
      YEAR = {2017},
    NUMBER = {3},
     PAGES = {371--383},
      ISSN = {1609-3321,1609-4514},
   MRCLASS = {11R34 (11R37 14G32)},
  MRNUMBER = {3711002},
MRREVIEWER = {Th\cfac ong\ Nguy\cftil en-Quang-\Dbar\cftil o},
       DOI = {10.17323/1609-4514-2017-17-3-371-383},
       URL = {https://doi.org/10.17323/1609-4514-2017-17-3-371-383},
}

@article {shimizu,
    AUTHOR = {Shimizu, Ryoji},
     TITLE = {The {N}eukirch-{U}chida theorem with restricted ramification},
   JOURNAL = {J. Reine Angew. Math.},
  FJOURNAL = {Journal f\"ur die Reine und Angewandte Mathematik. [Crelle's
              Journal]},
    VOLUME = {785},
      YEAR = {2022},
     PAGES = {187--217},
      ISSN = {0075-4102,1435-5345},
   MRCLASS = {11R32 (11R37 14G32)},
  MRNUMBER = {4402495},
MRREVIEWER = {Cornelius\ Greither},
       DOI = {10.1515/crelle-2021-0090},
       URL = {https://doi.org/10.1515/crelle-2021-0090},
}

@article {hoshi,
    AUTHOR = {Hoshi, Yuichiro},
     TITLE = {On the field-theoreticity of homomorphisms between the
              multiplicative groups of number fields},
   JOURNAL = {Publ. Res. Inst. Math. Sci.},
  FJOURNAL = {Publications of the Research Institute for Mathematical
              Sciences},
    VOLUME = {50},
      YEAR = {2014},
    NUMBER = {2},
     PAGES = {269--285},
      ISSN = {0034-5318,1663-4926},
   MRCLASS = {11R04 (12E99)},
  MRNUMBER = {3223474},
MRREVIEWER = {Art\=uras\ Dubickas},
       DOI = {10.4171/PRIMS/133},
       URL = {https://doi.org/10.4171/PRIMS/133},
}

@book {poitou,
    AUTHOR = {Poitou, Georges},
     TITLE = {Cohomologie galoisienne des modules finis},
    SERIES = {Travaux et Recherches Math\'ematiques},
    VOLUME = {No. 13},
      NOTE = {S\'eminaire de l'Institut de Math\'ematiques de Lille, sous la
              direction de G. Poitou},
 PUBLISHER = {Dunod, Paris},
      YEAR = {1967},
     PAGES = {xviii+279},
   MRCLASS = {18.20 (10.00)},
  MRNUMBER = {219591},
}

@article {niibo,
    AUTHOR = {Niibo, Hirofumi},
     TITLE = {Id\`elic class field theory for 3-manifolds},
   JOURNAL = {Kyushu J. Math.},
  FJOURNAL = {Kyushu Journal of Mathematics},
    VOLUME = {68},
      YEAR = {2014},
    NUMBER = {2},
     PAGES = {421--436},
      ISSN = {1340-6116,1883-2032},
   MRCLASS = {57M27 (11R37 11S31)},
  MRNUMBER = {3243372},
MRREVIEWER = {W.\ Narkiewicz},
       DOI = {10.2206/kyushujm.68.421},
       URL = {https://doi.org/10.2206/kyushujm.68.421},
}

@misc{Xu26,
  author        = {Xiaoyu Xu},
  title         = {Profinite rigidity in lattices of $\mathrm{PSL}(2,\mathbb{C})$},
  year          = {2026},
  eprint        = {2608.07350},
  archivePrefix = {arXiv},
  primaryClass  = {math.GT},
  doi           = {10.48550/arXiv.2608.07350},
}

\vspace{0.4pc}           
\enlargethispage{3\baselineskip}

\makeatletter
\def\@setaddresses{%
  \par\nobreak\begingroup \scriptsize
  \def\author##1{\par\nobreak\addvspace{3pt}{\scshape##1}\par\nobreak}%
  \def\\{\unskip, \ignorespaces}%
  \interlinepenalty\@M
  \def\address##1##2{%
    \begingroup\par\nobreak\addvspace{2pt}{\scshape##1}\par\nobreak\indent
    \@ifnotempty{##1}{(\ignorespaces ##1\unskip )}%
      {\scshape\ignorespaces ##2}%
    \par\endgroup}%
  \def\curraddr##1##2{%
    \begingroup \@ifnotempty{##2}{%
      \nobreak\indent \curraddrname\@ifnotempty{##1}{, \ignorespaces ##1\unskip }\/:\space ##2\par}%
    \endgroup}%
  \def\email##1##2{%
    \begingroup\@ifnotempty{##2}{%
      \nobreak \indent \emailaddrname
      \@ifnotempty{##1}{, \ignorespaces ##1\unskip }\/:\space \ttfamily ##2\par}%
    \endgroup}%
  \def\urladdr##1##2{%
    \begingroup \def~{\char`\~}%
    \@ifnotempty{##2}{%
      \nobreak \indent \urladdrname
      \@ifnotempty{##1}{, \ignorespaces ##1\unskip }\/:\space \ttfamily ##2\par }%
    \endgroup }%
  \addresses\endgroup}
\makeatother


\end{document}